\documentclass[12pt, reqno]{amsart}
\usepackage{amscd,amssymb,enumerate, amsmath,mathrsfs, amsfonts}
\usepackage{url}
\usepackage{tikz}
\usepackage{fullpage}

\usetikzlibrary{decorations.pathreplacing}
\usepackage[backref=page]{hyperref}
\usepackage{hyperref}

\usepackage[utf8]{inputenc}

\newtheorem{theorem}{Theorem}[section]

\newtheorem{lemma}[theorem]{Lemma}
\newtheorem{proposition}[theorem]{Proposition}
\newtheorem{corollary}[theorem]{Corollary}

\theoremstyle{definition}

\newtheorem{example}[theorem]{Example}

\newtheorem{remark}[theorem]{Remark}

\theoremstyle{remark}

\DeclareFontFamily{U}{wncy}{}
\DeclareFontShape{U}{wncy}{m}{n}{<->wncyr10}{}
\DeclareSymbolFont{mcy}{U}{wncy}{m}{n}
\DeclareMathSymbol{\Sh}{\mathord}{mcy}{"58}

\if false

\fi

\newcommand{{\ZZ}}	{\mathbb{Z}}

\renewcommand{\le}      {\leqslant}

\def\mydate{\number\day\space\ifcase\month \or January\or February\or March\or 
April\or May\or June\or July\or
August\or September\or October\or November\or December\fi \space\number\year}

\DeclareFontFamily{U}{wncy}{}
\DeclareFontShape{U}{wncy}{m}{n}{<->wncyr10}{}
\DeclareSymbolFont{mcy}{U}{wncy}{m}{n}
\DeclareMathSymbol{\Sh}{\mathord}{mcy}{"58}

\if false
\numberwithin{equation}{section}

\makeatletter
    \let\c@equation\c@thm
    \let\c@figure\c@thm
   \let\c@table\c@thm
\makeatother
\fi

\numberwithin{equation}{section}

\begin{document}


%
%

\title{Integer Dynamics}

\author[Dino Lorenzini]{Dino Lorenzini}
\address{Department of Mathematics, University of Georgia, Athens, GA 30602, USA}
\email{lorenzin@uga.edu, mentzmel@gmail.com, arvind.suresh25@uga.edu, makoto.suwama@gmail.com, hywangn@uga.edu}


\author[Mentzelos Melistas]{Mentzelos Melistas}

\author[Arvind Suresh]{Arvind Suresh}

\author[Makoto Suwama]{Makoto Suwama}
 
\author[Haiyang Wang]{Haiyang Wang}

\subjclass[2000]{11A67, 11B25, 11D45}


\begin{abstract} Let $b \geq 2$ be an integer, and 
 write the base $b$ expansion of any non-negative integer $n$ as $n=x_0+x_1b+\dots+ x_{d}b^{d}$, with $x_d>0$ and $ 0 \leq x_i < b$ for $i=0,\dots,d$.
Let $\phi(x)$ denote an integer polynomial such that $\phi(n) >0$ for all $n>0$.
Consider the map $S_{\phi,b}: {\mathbb Z}_{\geq 0} \to {\mathbb Z}_{\geq 0}$, with $ S_{\phi,b}(n) := \phi(x_0)+ \dots + \phi(x_d)$.
 It is known that the orbit set $\{n,S_{\phi,b}(n), S_{\phi,b}(S_{\phi,b}(n)), \dots \}$ is finite for all $n>0$. Each orbit contains a finite cycle, and 
for a given $b$, the union of such cycles over all orbit sets is finite.

Fix now an integer $\ell\geq 1$ and let $\phi(x)=x^2$. 
We show that the set of bases $b\geq 2$ which have at least one cycle of length $\ell$ always contains an arithmetic progression and thus has positive lower density. We also show that 
a 1978 conjecture of Hasse and Prichett on the set of bases with exactly two cycles needs to be modified,  raising the possibility that
this set  might not be finite.
\end{abstract}

\maketitle


\begin{section}{Introduction}
 
 Fix an integer $b \geq 2$.
 Any non-negative integer $n$ can be written uniquely in base $b$ as $n=x_0+x_1b+\dots+ x_{d}b^{d}$ with $x_d>0$ and $ 0 \leq x_i < b$ for $i=0,\dots,d$.
 We let $n=[x_0,\dots,x_d]_b$ denote the base $b$ expansion of $n$.
 Fix now a function $\phi: \{0,1,\dots, b-1\} \to {\mathbb Z}_{\geq 0}$,
  and consider the map $S_{\phi,b}: {\mathbb Z}_{\geq 0} \to {\mathbb Z}_{\geq 0}$, with 
  $$ S_{\phi,b}(n) := \phi(x_0)+ \dots + \phi(x_d).$$
For instance, when $b=10$ and $\phi(x)=x^2$, then $S_{\phi,b}(345)=3^2+4^2+5^2$.

The ordered sequence $[n,S_{\phi,b}(n), S_{\phi,b}(S_{\phi,b}(n)), \dots ]$ is called the   {\it orbit} of $n$ under $S_{\phi,b}$. We say that $n$ has {\it finite orbit} under $S_{\phi,b}$
if the set $\{n,S_{\phi,b}(n), S_{\phi,b}(S_{\phi,b}(n)), \dots \}$ is finite.   
\if false
B.\ M.\ Stewart \cite[Theorem 1]{Ste}) proved that {\it
there exists a constant $\gamma$, depending on $\phi$, such that if $n>\gamma$, then $ n > S_{\phi,b}(n)$.
In particular, every positive integer $n$ has finite orbit under $S_{\phi,b}$.}
\fi
Any finite orbit contains a  finite {\it cycle},   a non-empty sequence of integers $ {\rm cyc}(n_1,\dots, n_\ell)$ such that $S_{\phi,b}(n_\ell)=n_{1}$ 
and when $\ell>1$, $S_{\phi,b}(n_i)=n_{i+1}$ for $i=1,\dots, \ell-1$.
This cycle of {\it length} $\ell$ is unique up to cyclic permutation of its terms. 

\begin{example}
Let $\phi(x)=x^2$ and $b=30$. Take the integer $c:=315181514012209182119$, which is obtained from {\it coronavirus}\footnote{This 
project was completed in the early months of the Covid-19 pandemic. Meetings were held using Zoom.} 
by substituting each letter with its position in the alphabet. The repeated use of the function $S_{\phi,b}$ quickly brings $c$ to a standstill:  the orbit of $c$ is $[c, 4189, 738, 900, 1, 1, \dots ]$, and the cycle associated to this orbit is $[1]$. On the other hand, when $b=2020$, the orbit of $c$ has length $953$
with a cycle of length $801$. The outlook for $2021$ is no better, as both the length of the orbit and the cycle length increase.
\end{example}

B.\ M.\ Stewart \cite[Theorem 1]{Ste} proved that {\it
there exists a constant $\gamma$, depending on $\phi$, such that if $n>\gamma$, then $ n > S_{\phi,b}(n)$.
In particular every positive integer $n$ has finite orbit under $S_{\phi,b}$.}
It follows from Stewart's Theorem 
that the orbits of $S_{\phi,b}$ produce only finitely many distinct cycles. We will call the set of distinct cycles associated with the orbits of $S_{\phi,b}$ {\it the cycles associated with $S_{\phi,b}$}. 
The complete determination of the cycles of a given  $S_{\phi,b}$ is   computationally quite expensive for large $b$'s. 
When $\phi(x) = x^m$, one can proceed as follows. 

\begin{theorem} 
\label{cor.digits}
Let $\phi(x) = x^m$.
For each integer $n \leq  (m-1)b^m-1$, compute the cycle of $S_{\phi,b}(n)$. Then the union of all these cycles is the complete set of cycles associated 
with $S_{\phi,b}$.
\end{theorem}

This statement follows from \cite[Theorem 7]{Ste}, and the key to Stewart's proof is that if $n> (m-1)b^m-1$, then $n> S_{\phi,b}(n)$.  
Hence, {\it  every cycle for  $S_{\phi,b}$ contains a positive integer at most equal to $(m-1)b^m-1$.
In particular, every cycle for  $S_{\phi,b}$ contains a positive integer whose base $b$ expansion has at most $m+1$ digits. }

When $\phi(x)=x^2$, the number of $1$-cycles of a given $S_{x^2,b}$ is explicitly determined by the following theorem of P.\ Subramanian \cite[Theorem 1.2]{Sub}
(see also \cite[Section 3]{H-P}, and Proposition \ref{pro.1cyle}).  Recall that a divisor $d$ of a positive integer $n$ is called {\it proper} if $1\leq d<n$.

\begin{theorem} 
\label{thm.Sub} \label{cor.thin}
The number of 
$1$-cycles of  $S_{x^2,b}$ is equal to the number of proper divisors of $b^2+1$. 
\end{theorem}

 For convenience, let us call $[1]$ the {\it trivial} cycle of $S_{x^m,b}$.
 Let $\ell \geq 1$ be any integer. Let $B(\ell)$ denote the set of bases $b$ such that $S_{x^2,b}$  has at least one non-trivial cycle of length $\ell$. 
 Theorem \ref{thm.Sub} 
 implies that the natural density of $B(1)$ is $1$ (see Remark \ref{emp.naturaldensityB(1)}). It is not hard to show that $B(\ell)$ is infinite for all $\ell$ (see Example \ref{ex.cycleslengthell}), 
 and it is natural to wonder whether $B(\ell)$ has a positive natural density.  
 Let $S \subset {\mathbb N}$ be any subset. Let ${\displaystyle S(n):=\{1,2,\ldots ,n\}\cap S}$ and ${\displaystyle s(n):=|S(n)|}$.
Recall that the {\it lower density }  $ \underline{d}(S)$   of $S$  is defined as $ \underline{d}(S) := \liminf_{n \rightarrow \infty} \frac{ s(n) }{n}$.
\if false
Similarly, the upper density $\overline{d}(S) $ of $S$ is given by
$ \overline{d}(S) := \limsup_{n \rightarrow \infty} \frac{s(n)}{n} $.
If both $ \underline{d}(S)$ and $\overline{d}(S) $ exist and are equal, the {\it natural density} $d(S)$ of $S$ is defined as $d(S):=  \underline{d}(S)$.
\fi
In this article, we show:
 
\medskip
\noindent
{\bf Theorem (see  \ref{cor.lowerbound} and \ref{pro.3-4percent}).}
{\it Let $\ell \geq 1$. Then $B(\ell)$ has a positive lower density. More precisely, $B(\ell)$ always contains an explicit arithmetic progression.
The sets $B(2)$, $B(3)$, and $B(4)$ have lower density bounded below by $0.57$, $0.21$, and $0.11$, respectively.
}
\smallskip

The key ingredient in the proof of Theorem \ref{cor.lowerbound}
is the existence of special   $\ell$-cycles that we now define.
Let $c={\rm cyc}(n_1,\dots,n_\ell)$ denote a cycle of length  $\ell$ for $S_{x^2,b}$. We say that $c$ is a {\it propagating cycle} if (i) every integer $n_i$, $i=1,\dots, \ell$, has at most two digits when written in base $b$ and (ii) $b$ does not divide $n_i$, for all $i=1\dots, \ell$.  The name `propagating' is justified by our next theorem. It is easy to check with  Theorem \ref{cor.digits} that all   $1$-cycles of $S_{x^2,b}$  
are propagating.

\smallskip
\noindent
{\bf Theorem (see \ref{thm.explicit line}).}
{\it Let $b_0 \geq 2$. Assume that $S_{x^2,b_0}$ has $s$ distinct propagating cycles, 
of lengths $\ell_1, \dots, \ell_s$, respectively (repetitions are allowed). Let $t \geq 0$ be any integer, and let 
$b:=b_0+t(b_0^2+1)$. Then $S_{x^2,b}$ has (at least) $s$ distinct propagating cycles, of lengths $\ell_1, \dots, \ell_s$ respectively.
}
\smallskip

Propagating $\ell$-cycles can naturally be seen as corresponding to integer points on an algebraic variety $V_\ell/{\mathbb Q}$. It turns out
that $V_\ell$  has the property that, through every integer point on it corresponding to an $\ell$-cycle, there passes at least one integer line
given by explicit equations. This arithmetico-geometrical fact underlies the proof of Theorem  \ref{thm.explicit line}.
When $\ell=1$ or $2$, there is in addition a second integer line passing through each point, which also propagates cycles. We exploit 
the existence of this second line when $\ell=2$ in Proposition \ref{thm.HPB} and Remark \ref{rem.data1}.
 \if false
\smallskip
\noindent
{\bf Theorem (see \ref{thm.HPB}).}
{\it 
There exist infinitely many integers $b\geq 2 $ such that $S_{x^2,b}$ has exactly two non-trivial $1$-cycles, but 
  $ b \notin L(x^2,3)$ because $S_{x^2,b}$ also has a $2$-cycle.

The Bouniakovsky Conjecture implies that there exist infinitely many integers $b\geq 2 $ such that $S_{x^2,b}$ has no non-trivial $1$-cycles,
but $b \notin L(x^2,2)$ because $S_{x^2,b}$ also has two distinct $2$-cycles.
}
\smallskip
\fi

 \if false
 Our computations in the case of $\phi(x)=x^2$ indicate that 
 the Hasse--Prichett conjecture concerning the exact description of the set $L(x^2,2)$ needs to be modified, as we found that $8626 \in L(x^2,2)$. 
 When $b=8626$, $S_{x^2,b}$ has exactly two cycles, $[1]$ and a second one of length $4240$, generated by $n=30$. 
 
 Cycles for $S_{x^2,b}$ seem to be so numerous that one might wonder whether the following stronger question might have a positive answer. 
 Let $d \geq 0$ and $i \geq 1$ be any positive integers; {\it is the set $L(x^2,i,d)$ of all bases $b$ such that $S_{x^2,b}$ has exactly $i$ cycles of length greater than $d$ finite?} With this notation, we find that  $L(x^2,i)=L(x^2,i,0)$. 
 \fi

Some of Stewart's 1960 results \cite{Ste} have been independently rediscovered by H. Hasse and G. Prichett in 1978 (\cite[Theorem 4.1]{H-P}).
At the end of \cite{H-P},  Hasse and  Prichett propose the following conjecture:

\smallskip
 {\it Let $\phi(x)=x^2$ and consider the set $L(x^2,2)$ of all integers $b \geq 2$ 
such that the list of cycles associated with $S_{\phi,b}$ consists of the trivial cycle $[1]$ and exactly one additional cycle. Then $L(x^2,2)=\{6,10,16,20,26,40\}$.}
\smallskip

Hasse and Prichett made this conjecture after having numerically verified it for $b \leq 500$.

Let $\phi(x)$ be any polynomial taking positive values on ${\mathbb Z}_{> 0}$.
Let $L(\phi,i)$ denote the set of integers $b \geq 2$ such that the list of cycles associated with $S_{\phi,b}$ consists  of exactly $i$ distinct  cycles. 
It is natural to wonder whether the Hasse--Prichett conjecture for $\phi(x)=x^2$ and $i=2$ is in fact only a specific instance of a much more general
 phenomenon, namely,
 that all the sets $L(\phi,i)$ are finite, for all $i \geq 1$. 

  Curiously, Hasse and Prichett do not mention in \cite{H-P} a similar conjecture for the set $L(x^2,1)$. In this case, that $L(x^2,1)=\{2,4\}$ seems to be by now a  folklore  conjecture.
\if false 
  It is now common to refer to an integer $n$ such that the orbit of $n$ under $S_{\phi,b}$ contains $1$ as a {\it happy number} in base $b$. The earliest  appearance of this terminology might be in 1994 in the second edition of Richard Guy's  {\it Unsolved Problems in Number Theory} \cite[Problem E34]{Guy} in the case of $b=10$. The conjecture $L(x^2,1)=\{2,4\}$ can thus be rephrased as stating that {\it the only bases $b \geq 2$ such that every positive integer is happy in base $b$ are $b=2$ and $4$}.
  \fi 
  It is stated in \cite{OEIS}, A161872, that the conjecture has been verified for all $b< 500,000,000$.
\if false
that 
{\it all positive integers are happy numbers in base 2 and base 4; they are called "happy bases". There are no other happy bases $< 500,000,000$.}}
\fi

Subramanian's Theorem \ref{cor.thin} shows that if the set $L(x^2,2)$  
is infinite, it will indeed be very sparse, since 
if $b \in L(x^2,1)$ or $ L(x^2,2)$, then  $b^2+1$ is prime.
To justify this claim, note that Theorem \ref{thm.Sub} implies that $b^2+1$ can only have at most one proper  divisor bigger than $1$. This can happen only when $b^2+1=p^2$ for some prime $p$. 
But the factorization $1=(p-b)(p+b)$ has no integer solutions when $b \geq 2$.

The unboundedness of the set of integers $b$ such that  $b^2+1 $  is prime is implied by a general 1857 conjecture
of Victor Bouniakowsky \cite[page 328]{Bou}, that any irreducible polynomial $f(x) \in {\mathbb Z}[x]$ with positive leading coefficient
takes infinitely many prime values if  the values   $f( 1 )$, $f ( 2 )$, $f( 3 )$, $\dots$ have no common factor.
 This conjecture in the case of $f(x)=x^2+1$ was one of E. Landau's four problems 
 presented at the 1912 International Congress of Mathematicians (see 
 \cite{HL}, pp 46-48). 
 Note that a {\it negative} answer to the Hasse--Prichett Conjecture 
 (in the strong sense where $L(x^2,2)$ would be proved to be infinite) would provide a positive answer to Landau's problem. 
 
  The computations below were done using the cluster Sapelo2 at the Georgia Advanced Computing Resource Center.
 
\begin{theorem} 
\label{thm.data0}
Let $b \leq 1000000$. If $b \in L(x^2,2)$, then 
$$ b \in  \{6,10,16,20,26,40, 8626, 481360 \}.$$
\end{theorem}
Thus the Hasse--Prichett Conjecture at the very least needs to be modified to include the bases $b=8626$ and $b=481360 $.
The existence of the large gap between these two bases might be seen as evidence against  the validity of the modified conjecture. 

Searching for other types of finiteness, one may wonder for instance whether, for a given integer $d$, the set $M(x^2,d)$ of all bases $b$ such that 
$S_{x^2,b}$ only has cycles of length at most $d$ is finite. We have $\{2,4\} \subseteq M(x^2,1)$ and $\{2,3,4,13,18,92\} \subseteq M(x^2,2)$.
\begin{theorem}  \label{rem.data0extended}
Let $400 \leq b \leq 1100000$. If $b \in M(x^2,10)$, then 
$$b \in \{432,
596,
687,1068,
1932, 3918,  288504\}.$$
\end{theorem}
  Moreover, $452808 \in M(x^2,12)$. The presence of the large gap in $ M(x^2,10)$ might be seen as evidence that $ M(x^2,10)$ might be infinite. 
\if false
 Full list
  all integers 2-19
25
26
29
30
32
37
38
40
46
50
54
61
68
88
92
130
132
144
173
178
218
268
356
362
370
378
\fi

For comparison with the case $\phi(x)=x^2$, let us note the following results for $\phi(x)=x^3$.  

\smallskip
\noindent
{\bf Theorem (see \ref{pro.cubic1cycles} and \ref{pro.cubic1cyclebsquare}).}  
{\it Let $b >2$ be a square, or an integer that is not divisible by $9$. Then $S_{x^3,b}$ has at least one non-trivial $1$-cycle.
In particular, the set  of bases $b\geq 2$ such that $S_{x^3,b}$ has a non-trivial $1$-cycle has lower density bounded below by $8/9$.
 }
 \smallskip

 \smallskip
\noindent
{\bf Theorem (see 
\ref{pro.2cycleCubic} and \ref{cor.5cycles}).} Let $k \geq 1$ be any integer.
{\it
  \begin{enumerate}[\rm (a)]
  \item Let $b=3k+1$. Then $S_{x^3,b}$ has at least five distinct cycles.
 \item Let $b=9k^2 + 15k + 7$ or $9k^2 + 21k + 13$. Then  $S_{x^3,b}$ admits at least one $2$-cycle.
 \end{enumerate}
 }
 \smallskip
 Part (a) of the above theorem can be interpreted as saying that at least $1/3$ of the integers do not belong to  $  L(x^3,i)$ with $i\leq 4$.
  In the spirit of the Hasse--Prichett conjecture, we offer  the following questions in the case where 
 $\phi(x)=x^3$: 
 do   the equalities $L(x^3,1)=\{2\}$, $L(x^3,2)=\emptyset$,  $L(x^3,3)=\{3,26\}$, and $L(x^3,4)=\{5,90,188\}$ hold?

 The authors gratefully acknowledge computing support from the Georgia Advanced Computing Resource Center, and financial support from the Research and Training Group in Algebraic Geometry, Algebra, and Number Theory (AGANT)
 at the University of Georgia. 
 The authors also thank the referee for a thoughtful report.
 \end{section}

\begin{section}{Propagating $\ell$-cycles} \label{sec.propagating}

Let $\phi(x)$ be any polynomial taking positive values on ${\mathbb Z}_{> 0}$, and let $\ell \geq 1$ be any integer. It is natural to wonder whether 
there exist bases $b \geq 2$ such that $S_{\phi,b}$ admits  cycles of length $\ell$.  We consider this question in this section mainly 
when $\phi(x)=x^2$. We start with some general observations for $\phi(x)=x^m$, $m \geq 2$.

\begin{example} \label{ex.cycleslengthell} Let $\ell > 1$ and $m \geq 2$ be any integers. Let $b:=c^{m^{\ell}-1}$. Then the orbit of $c^m$ under $S_{x^m,b}$ is a cycle of length $\ell$,
namely the cycle ${\rm cyc}(c^m,c^{m^2},\dots, c^{m^{\ell}}=bc)$.
Thus we can  quantify the infinitude of the set of bases $b$ such that $S_{x^m,b}$ contains a cycle of length $\ell$  by noting that this set contains 
all integer values of the polynomial $f(t)=t^{m^{\ell}-1}$ when $t>1$. We show in Corollary \ref{cor.lowerbound} that when $m=2$, the same statement holds
with a polynomial $f(t)$ of degree $1$.
\end{example}

\if false
Since in the above example, the same base $b:=c^{m^{\ell}-1}$ might have several distinct cycles of length $\ell$, one may wonder whether 
different examples could be produced with $t^{m^{\ell}-1}$ replaced by a polynomial 
$f(t) \in {\mathbb Z}[t]$ of much smaller degree. 
\fi
\begin{example} \label{ex.2-cyclexm} 
Take a prime $p>m$, and let $\ell $ denote the order of $m$ in $({\mathbb Z}/p{\mathbb Z})^*$. 
We find that the base $b=c^p$ admits  the orbit of $c^m$ under $S_{x^m,b}$ as a cycle of length $\ell$. Thus, when the order of $m$ equals $\ell:=p-1$, 
we find that all integer values of the polynomial $f(t)=t^{\ell+1}$ when $t>1$ are contained in the set of bases $b$
such that $S_{x^m,b}$ contains a cycle of length $\ell$.  

For instance
when $\ell=4$ and $m=2$, we have for all bases $b=c^5$ the $4$-cycle ${\rm cyc}(c^2, c^4, c^8, c^6)$. 
        When $\ell=2$, we can consider the $2$-cycle ${\rm cyc}(c^m, c^{m^2})$ in base $b=c^{m+1}$, since in base $b$, $c^{m^2} = c (c^{m+1})^{m-1}$. 
\end{example}

\begin{example} \label{ex.existencepropagating}
{\it Existence of propagating $\ell$-cycles.} None of the examples of $\ell$-cycles exhibited above when $\phi(x)=x^2$ are examples of propagating cycles 
as defined in the introduction,
since although each integer in the cycle has at most two digits when written in base $b$, at least one integer in the cycle is divisible by $b$. Our next example
is an example of a propagating cycle.

Fix $\ell \geq 2$. Choose coprime positive  integers $\alpha$ and $ \beta$ such that $\beta$ divides $\alpha^{2^{\ell}}-\alpha$, and set $\gamma:=\alpha^2+\beta^2$. For instance, one can choose $\alpha=\beta=1$ and $\gamma=2$. Set $b:=(\gamma^{2^{\ell-1}}-\alpha)/\beta$. 
Since $\beta $ divides $ \alpha^{2^{\ell}}-\alpha$ and $\alpha^2=\gamma - \beta^2$, we find that $b$ is an integer. 
Then
$${\rm cyc}(\gamma, \gamma^2, \gamma^4, ..., \gamma^{2^{\ell-1}} = \alpha+\beta b)$$
is an $\ell$-cycle for $S_{x^2,b}$. Indeed, it is easy to verify that $\alpha, \beta < b$, so that $\gamma^{2^{\ell-1}}= [\alpha,\beta]_b$.

It is easy to check that 
$b$ does not divide any of the integers $\gamma^{2^{k}}$ for $k=0,\dots ,\ell-1$. Since $\gamma^{2^{\ell-1}}$ has only two digits in base $b$,  the smaller integers $\gamma^{2^k}$ have at most two digits.  Hence, this cycle is a propagating cycle. 
\end{example}

Let $\ell \geq 1$. 
Consider the affine space ${\mathbb A}^{2\ell +1}$ and label its coordinates by 
$$b, \text{ and } x_{i}, y_{i}, \text{ for } i=1, \dots, \ell.$$ 
Let $V_\ell$ denote the algebraic subvariety of  ${\mathbb A}^{2\ell +1}$ defined by 
$x_1^2+y_1^2=x_1+b y_1$ when $\ell=1$, and in general by the $\ell$ equations
\begin{equation*}
\begin{array}{ll}
 x_{i}^2+y_i^2= x_{i+1}+by_{i+1} & \text{ for } i=1, \dots, \ell-1, \text{ and } \\
 x_{\ell}^2+y_\ell^2= x_{1}+by_{1}. &
 \end{array}
 \end{equation*}

Let now ${\rm cyc}(n_1,\dots,n_\ell)$ be a propagating $\ell$-cycle for $S_{x^2,b_0}$. 
We will use below the notation $x_i, y_i$ to denote variables, and use $\mathbf{x}_i,\mathbf{y}_i$ to denote specific values.
With this notation, we write $n_i=[\mathbf{x}_i,\mathbf{y}_i]_{b_0}$ with $0\leq \mathbf{x}_i,\mathbf{y}_i\leq b_0-1$ and $\mathbf{x}_i \neq 0$.   
Then the integer point $(b_0,\mathbf{x}_i,\mathbf{y}_i, i=1,\dots, \ell)$ satisfies the equations of the variety $V_\ell$, and thus 
{\it a propagating $\ell$-cycle for $S_{x^2,b_0}$ corresponds to an integer point on the variety $V_\ell $ that satisfies the added requirement that $0\leq \mathbf{x}_i,\mathbf{y}_i\leq b_0-1$ and $\mathbf{x}_i \neq 0$}. 

\begin{theorem} 
\label{thm.mainpropagating}
Let $P:=(b_0,   \mathbf{x}_1,\mathbf{y}_1,  \dots, \mathbf{x}_\ell,\mathbf{y}_\ell)$ be any point on the variety $V_\ell$, with coordinates in ${\mathbb C}$. 
Then there exists a line in ${\mathbb A}^{2\ell +1}$ that passes through $P$ and is fully contained in $V_\ell$. This line can be given 
by the parametric equations 
\begin{equation} \label{eq.param}
\begin{array}{lcl}
    b(t)&:=& b_0 + (b_0^2+1)t, \\
    x_i(t)&:= & \mathbf{x}_i + (\mathbf{x}_ib_0 - \mathbf{y}_i)t ,\\
    y_i(t)& :=& \mathbf{y}_i + (\mathbf{y}_ib_0 + \mathbf{x}_i)t,   \text{ for } i=1,\dots, \ell.
    \end{array}
\end{equation}
If $P$ is a point with integer coordinates which corresponds to a propagating $\ell$-cycle of $S_{x^2,b_0}$, then for every non-negative integer $\mathbf{t}$, the integer point
$$P(\mathbf{t}):=(b(\mathbf{t}), x_1(\mathbf{t}), y_1(\mathbf{t}), \dots, x_\ell(\mathbf{t}), y_\ell(\mathbf{t}))$$
corresponds to a propagating $\ell$-cycle for $S_{x^2,b(\mathbf{t})}$.
\end{theorem} 
\begin{proof}
The proof of the theorem is not difficult, once the parametric equations \eqref{eq.param} are available. 
Indeed, one finds that in ${\mathbb C}[t]$,
$$
\begin{array}{rcl}
x_i(t)^2+y_i(t)^2-x_{i+1}(t)-b(t)y_{i+1}(t) & = &
 (\mathbf{x}_i^2+ \mathbf{y}_i^2-\mathbf{x}_{i+1}-b_0 \mathbf{y}_{i+1})((b_0 t+1)^2 + t^2) \\
 &=& 0,
 \end{array}
$$
for all $i=1,\dots,\ell$ (where the index  $\ell+1$ is set to mean index $1$).

Assume now that $P$ corresponds to a propagating $\ell$-cycle of $S_{x^2,b_0}$.
We further need to show that if $\mathbf{t} \geq 0$, then $0 \leq x_i(\mathbf{t}), y_i(\mathbf{t}) < b(\mathbf{t})$, and $x_i(\mathbf{t})\neq 0$.
These inequalities follow from the fact that since $\mathbf{x}_i \neq 0$ and $\mathbf{y}_i < b_0$, we have $\mathbf{x}_ib_0 - \mathbf{y}_i>0$, $ \mathbf{x}_i +b_0 \mathbf{y}_i >0$,  and  
$(\mathbf{x}_ib_0 - \mathbf{y}_i), (\mathbf{x}_i +b_0 \mathbf{y}_i) \leq b_0^2-1$.
\end{proof}

 \begin{remark} \label{rem.reduced} Consider an $\ell$-cycle ${\rm cyc}(n_1,\dots, n_\ell)$ as in Theorem \ref{thm.mainpropagating}, with 
 $n_i:=\mathbf{x}_i + \mathbf{y}_ib_0$.   
 Let $g:=\gcd(b_0^2+1,n_1,\dots, n_\ell)$.
 Given the parametrization \eqref{eq.param}, we will call the following parametrization of the same line the {\it reduced integer parametrization} of the line:
 $$\begin{array}{lcl}
    b(t)&:=& b_0 + t(b_0^2+1)/g, \\
    x_i(t)&:= & \mathbf{x}_i + t(\mathbf{x}_ib_0 - \mathbf{y}_i)/g ,\\
    y_i(t)& :=& \mathbf{y}_i + t(\mathbf{y}_ib_0 + \mathbf{x}_i)/g,   \text{ for } i=1,\dots, \ell.
    \end{array}
   $$
 Note that $g$ divides  $\mathbf{x}_ib_0 - \mathbf{y}_i$ since $\mathbf{x}_ib_0 - \mathbf{y}_i= b_0(\mathbf{y}_ib_0 + \mathbf{x}_i)-\mathbf{y}_i (b_0^2+1) $.
 Note also that $g< b_0^2+1$ since if $g= b_0^2+1$, then $b_0^2+1$ would divide $n_i:=\mathbf{x}_i + \mathbf{y}_ib_0$, but this is not possible since 
 $n_i\leq b_0^2-1$.
\end{remark}

\begin{remark} For each $\mathbf{t} \in {\mathbb C}$, we can define an endomorphism $\varphi_\mathbf{t}: V_\ell \to V_\ell$ on the affine variety $V_\ell$
using the ring homomorphism $\varphi_\mathbf{t}^*$ on functions on $V_\ell$ defined as:
\begin{equation*} 
\begin{array}{lcl}
    \varphi_\mathbf{t}^*(b)&:=& b + \mathbf{t}(b^2+1), \\
    \varphi_\mathbf{t}^*(x_i)&:= & x_i + \mathbf{t}(x_ib-y_i),\\
   \varphi_\mathbf{t}^*(y_i)& :=& y_i + \mathbf{t}(y_ib+x_i),   \text{ for } i=1,\dots, \ell.
    \end{array}
\end{equation*}
\if false
Indeed, we find that 
$$
\begin{array}{rcl}
 \varphi_\mathbf{t}(x_i)^2+\varphi_\mathbf{t}(y_i)^2-\varphi_\mathbf{t}(x_{i+1})- \varphi_\mathbf{t}(b)\varphi_\mathbf{t}(y_{i+1}) & = &
 -( {x}_i^2+  {y}_i^2- {x}_{i+1}-b   {y}_{i+1})((b  \mathbf{t}+1)^2  + \mathbf{t}^2 ) \\
 &=& 0,
 \end{array}
$$
for all $i=1,\dots,\ell$.

Let $\mathbf{t} \neq 0$.  
\if false
Let $F$ denote the closed subset of $V_\ell$ where $(1+b  \mathbf{t})^2  + \mathbf{t}^2 =0$.
On $V_\ell \setminus F$, define  $\varphi_s$ by substituting 
 $s = -\mathbf{t}/ (b^2 \mathbf{t}^2 + 2 b \mathbf{t} + \mathbf{t}^2 +1) $ in the above formula (the formula for $\varphi_s$ thus involves rational functions on $V_\ell$ with denominator
 $b^2 \mathbf{t}^2 + 2 b \mathbf{t} + \mathbf{t}^2 +1 \in {\mathbb C}[b]$).
 Then the endomorphism $\varphi_\mathbf{t} $ satisfies  $\varphi_\mathbf{t} \circ \varphi_s= {\rm id}$ on $V_\ell \setminus F$.
\fi 
 The closed subset defined by $(b  \mathbf{t}+1)^2  + \mathbf{t}^2=0$ is the union of the two codimension $1$ subsets $F_1$ and $F_2$ defined by  
 $b= i - 1/ \mathbf{t}$ and $b= -i - 1/ \mathbf{t}$ 
 respectively, and their images  under $\varphi_\mathbf{t}$
 are the closed subsets defined by $b+i=0$ and $b-i=0$, respectively. These latter closed subsets are not irreducible in $V_\ell$, while for 
 most values of $\mathbf{t}$, the closed subsets 
 $F_1$ and $F_2$ 
 are irreducible. Thus in general, $\varphi_\mathbf{t}$ is not an automorphism.
 
 \if false
 In $V_1$, the closed subset $V(b-i)$ defined by $b=i$  is the union of the closed subsets $V(b-i,x+iy)$ and $V(b-i,x-iy-1)$, which intersect at the point
 $(i,1/2, i/2)$.
 When $t=i/2$, the endomorphism $\varphi_t: V_1 \to V_1$ contracts the closed subset $V(b-i,x-iy-1)$ to the point $(i,1/2, i/2)$.
 \fi
 \fi
\end{remark}

\begin{theorem} \label{thm.explicit line}
Let $b_0 \geq 2$. Assume that $S_{x^2,b_0}$ has $s$ distinct propagating cycles, 
of lengths $\ell_1, \dots, \ell_s$, respectively (repetitions are allowed). Let $\mathbf{t} \geq 0$ be any integer, and let 
$b:=b_0+\mathbf{t}(b_0^2+1)$. Then $S_{x^2,b}$ has (at least) $s$ distinct propagating cycles, of lengths $\ell_1, \dots, \ell_s$ respectively.
\end{theorem}
\begin{proof} When the integers $\ell_1, \dots, \ell_s$ are distinct, the statement of the theorem follows immediately from the existence of the 
`propagating' lines proved in Theorem \ref{thm.mainpropagating}. Suppose now that for some integer $\ell$, there are exactly $j>1$ indices $i$ such that $
\ell_i=\ell$. Since we start with $j$ distinct propagating $\ell$-cycles, we have $j$ distinct points on $V_\ell$, and Theorem \ref{thm.mainpropagating}
proved the existence of $j$ distinct lines on $V_\ell$. To conclude the proof of Theorem \ref{thm.explicit line}, it suffices to prove that these lines do not intersect in $V_\ell$ at a point where $\mathbf{t} $ is a positive integer. This can be checked directly. Assume that $\mathbf{t} \neq 0$, and that we have two $\ell$-cycles $(b_0, \mathbf{x}_1, \mathbf{y}_1, \dots)$ and $(b_0, \overline{\mathbf{x}_1}, \overline{\mathbf{y}_1}, \dots)$ with 
$$\begin{array}{ccc}
\mathbf{x}_i+(\mathbf{x}_i b_0-\mathbf{y}_i)\mathbf{t}& =& \overline{\mathbf{x}_i}+(\overline{\mathbf{x}_i} b_0-\overline{\mathbf{y}_i})\mathbf{t},\\
\mathbf{y}_i+(\mathbf{x}_i+\mathbf{y}_i b_0)\mathbf{t}& =& \overline{\mathbf{y}_i}+(\overline{\mathbf{x}_i} +\overline{\mathbf{y}_i} b_0)\mathbf{t},
\end{array}
$$
for $i=1,\dots, \ell$.
Then
$$\begin{array}{ccc}
(\mathbf{x}_i-\overline{\mathbf{x}_i})(1+b_0\mathbf{t})  & =& (\mathbf{y}_i-\overline{\mathbf{y}_i})\mathbf{t},\\
(\mathbf{y}_i-\overline{\mathbf{y}_i})(1+b_0\mathbf{t})  & =& -(\mathbf{x}_i-\overline{\mathbf{x}_i})\mathbf{t}.
\end{array}
$$
We must have $(\mathbf{x}_i-\overline{\mathbf{x}_i})=0$ and, hence, $(\mathbf{y}_i-\overline{\mathbf{y}_i})=0$, since otherwise, 
the above equations imply that $(b_0\mathbf{t}+1)^2+\mathbf{t}^2=0$. This latter equation is not possible when both $\mathbf{t}$ and $b_0$ are real, which we assume.
\end{proof}

Denote by $\operatorname{PB}(\ell)$ the set of bases $b \geq 2$ such that $S_{x^2,b}$ has a propagating $\ell$-cycle.

\begin{corollary} \label{cor.lowerbound} Let $\ell \geq 2$. Let $b_0:=2^{2^{\ell-1}}-1$. Then the set $\operatorname{PB}(\ell)$ contains an arithmetic progression, and has lower density bounded below by $2/(b_0^2+1)$.
\end{corollary}
\begin{proof} 
The existence of a propagating $\ell$-cycle ${\rm cyc}(\gamma, \gamma^2,\dots)$ for the base  $b_0:=2^{2^{\ell-1}}-1$ is established in Example \ref{ex.existencepropagating}. Since $\gamma=2$, we find that the greatest common divisor of the elements in the cycle is $2$. 
Theorem \ref{thm.mainpropagating} and Remark \ref{rem.reduced}  show the existence of an arithmetic progression $b(\mathbf{t})=b_0+\mathbf{t}(b_0^2+1)/2$ such that for every integer $\mathbf{t} \geq 1$, $S_{x^2,b(\mathbf{t})}$ has a propagating $\ell$-cycle. The natural density of the set of positive integers in an arithmetic progression 
$\{ a\mathbf{t}+b \mid \mathbf{t} \geq 0\}$ is $1/a$. 
\end{proof}
\begin{example} Every base $b$ whose last digit in base $10$ is $3$ or $8$ has a propagating $2$-cycle. This follows from the fact that in base $b_0=3$,
the $2$-cycle ${\rm cyc}(2,4)$ is propagating. 
\end{example}
 
\begin{remark}
\label{rem.arithprog}
For later use, we note here the following facts. 
Consider a set $S$ of positive integers which contains a union $U:=\bigcup_{i=1}^n \left(\bigcup_{j=1}^{r_i} \{ a_it+b_{ij} \mid t \geq 0\}\right)$ of arithmetic progressions.  
Then the lower density $\underline{d}(S)$ of $S$ satisfies  $\underline{d}(S) \geq d(U)$. 
When the $a_i$ are pairwise coprime, we find that 
$$d(U)=  1- \prod_{i=1}^n \left(1-\frac{r_i}{a_i}\right).$$
\end{remark}
\begin{remark}
\label{emp.naturaldensityB(1)} Let us show now that $\operatorname{PB}(1)$ has natural density $1$. More generally, 
let $f(x) \in {\mathbb Z}[x]$ be such that $f({\mathbb Z}_{\geq 0} ) \subseteq {\mathbb Z}_{\geq 0}$, and consider  $B:=\{b \in {\mathbb N} \mid  f(b) \text{ is not prime} \}$.
Subramanian's Theorem \ref{cor.thin} shows that $\operatorname{PB}(1) $ has the same natural density as the set $B$ when $f(x)=x^2+1$. 

Let $S$ denote the set of primes $p$ such that there exists $b_p \in {\mathbb N}$ with $f(b_p)$ divisible by $p$. The set $B$ then contains the arithmetic progression $b_p+pt$ for each $p \in S$. Thus the lower density of $B$ is bounded below
by the product $1- \prod_{p \in S} \left(1-\frac{1}{p}\right)$. The product $\prod_{p \in S} \left(1-\frac{1}{p}\right)$ converges to $0$ if and only if 
the sum $\sum_{p \in S} \frac{1}{p}$ diverges. When $f(x)=x^2+1$, the set $S$ consists of $2$ and all primes $p$ congruent to $1$ mod $4$. It follows from Dirichlet's theorem on primes in arithmetic progression that $\sum_{p \in S} \frac{1}{p}$ diverges, so that the density of $B$ is $1$ in this case.
\end{remark}

\begin{proposition} \label{pro.3-4percent}
For $\ell=2,3,4,5$, the lower density $\underline{d}(\operatorname{PB}(\ell))$ is bounded below as follows:
  \begin{center}
    \begin{tabular}{|c|c|c|c|c|} 
      \hline
      $\ell$ & $2$ & $3$ & $4$ & $5$ \\
      \hline
      $\underline{d}(\operatorname{PB}(\ell)) \geq $ &  $0.5763$ &  $0.2127$ & $0.1144$ & $0.0429$ \\
      \hline
       
    \end{tabular}
  \end{center}
\end{proposition}
\begin{proof}
The steps in the computations of the four lower bounds given in the above table are the same for each $\ell$. 
First compute a set $P$  of propagating cycles. In our case, we used all propagating cycles with $2 \leq b \leq 500$. 
For each propagating cycle ${\rm cyc}(n_1,n_2,\dots)$ in base $b$, 
Theorem \ref{thm.mainpropagating} produces integer lines, 
which we parametrize using their reduced integer parametrization (see Remark \ref{rem.reduced}). The slope of each line is 
$\lambda:=(b^2+1)/\gcd(b^2+1,n_1,n_2,\dots)$.

Thus we now have a set of explicit arithmetic progressions
of the form $b+\lambda t$ which are contained in $\operatorname{PB}(\ell)$.
To compute a lower bound for the lower density, we further prune the set of lines from all the lines which do not have prime power slope. 
The lower density of the set of arithmetic progressions with $\lambda$ a prime power can be easily bounded below since the slopes are all pairwise coprime
 and we can use the formula in Remark \ref{rem.arithprog} to obtain a lower bound for
$\underline{d}(\operatorname{PB}(\ell))$. 
Our computations produce the following data:
\if false
In our case, we initially used all propagating cycles with $2 \leq b \leq 300$. 
For each propagating cycle ${\rm cyc}(n_1,n_2,\dots)$ in base $b$, compute $\lambda:=(b^2+1)/\gcd(b^2+1,n_1,n_2,\dots)$ and consider the set $S$ of all $\lambda$ found this way.
Theorem \ref{thm.mainpropagating} produces integer lines, 
which we parametrize using their reduced integer parametrization (see \ref{rem.reduced}). Thus we now have a set of explicit arithmetic progressions
of the form $b+\lambda t$ which are contained in $\operatorname{PB}(\ell)$.
To compute a lower bound for the lower density, we further prune the set $S$ from all its elements that are not prime or a power of a prime. 
Let  $S_{p}$  denote the resulting set.  The lower density of the set of arithmetic progressions with $\lambda$ in $S_p$ can be easily bounded below since the elements in $S_p$ are all pairwise co-prime
 and we can use the formula in Remark \ref{rem.arithprog} to obtain
$\underline{d}(\operatorname{PB}(\ell)) \geq 1- \prod_{\lambda \in S_{p}}(1-1/\lambda)$.
Our computations produce the following data:
\begin{center}
    \begin{tabular}{|c|c|c|c|c|c|}  
      \hline
      $\ell$ & $1$ & $2$ & $3$ & $4$ & $5$ \\
      \hline
      $|P|$  & $1300$ & $580$ & $219$ & $113$ & $33$ \\
      \hline
      $|S|$  & $527$ &$325$ & $138$ & $81$ & $26$ \\
      \hline
      $|S_p|$  & $172$ & $95$ & $46$ & $26$ & $11$ \\
      \hline
      $\underline{d}(\operatorname{PB}(\ell)) \geq $ & $0.7593$ & $0.3274$ & $0.2026$ & $0.1005$ & $0.0392$ \\
      \hline      
    \end{tabular}
  \end{center}
  \fi

\begin{center}
    \begin{tabular}{|c|c|c|c|c|c|}  
      \hline
      $\ell$ & $1$ & $2$ & $3$ & $4$ & $5$ \\
      \hline
      $|P|$  & $2444$ & $1163$ & $391$ & $190$ & $77$ \\
      \hline
      $\underline{d}(\operatorname{PB}(\ell)) \geq $ & $0.8917$ & $0.3507$ & $0.2127$ & $0.1144$ & $0.0429$ \\
      \hline      
    \end{tabular}
  \end{center}
  The data when $\ell=1$ is only included for information, since we know already that $\underline{d}(\operatorname{PB}(1))=1$.
The lower bound for $\underline{d}(\operatorname{PB}(2))$
is improved to $\underline{d}(\operatorname{PB}(2)) \geq  0.5763$ in Proposition \ref{pro.percent}.  
\end{proof}

\begin{remark} Computations indicate that the first integer $b$ such that $S_{x^2,b}$ has a propagating $\ell$-cycle might be much smaller than $2^{2^{\ell-1}}-1$ when $\ell >2$. One may wonder whether the first such $b$ might even be bounded by a polynomial function in $\ell$.
Computations show that for each $\ell \leq 20$, there exists a basis $b \leq 1230$ with an $\ell$-cycle. 
\end{remark}

\begin{remark}
Among the first $39000$ bases $b$, only $1330$, or about $3.41\%$, do not have a $2$-cycle. 
The lower density of the set of bases $b$ with a $2$-cycle might thus be quite larger than $0.5763$, and   
it would be interesting to determine if it is actually equal to $1$. 
Among the first $25000$ bases $b$, $17155$, or about $68.62\%$, have a $3$-cycle. 
\end{remark}

\if false

For $20\leq  \ell \leq 28$, we indicate below 
a triple $[b,\ell,n]$ where $b$ is a basis with a propagating $\ell$-cycle and $n$ is an integer in the cycle.
\begin{equation*}
\left. \begin{array}{llll}
[1230, 20, 579826], & [3346, 21, 7262125], & [ 8495, 22, 24918833 ], & \text{[6541, 23, 3684682]}, \\
\text{[6663, 24, 6451625 ]}, & [5032, 25, 3065450], & \text{[1752, 26, 959465]}, & [7405, 27, 8762650], \\
\text{[3707, 28, 942125]}, & \text{[2181, 30, 1190218]}. & 
\end{array}
\right.
\end{equation*}

We have not found a basis $b$ with a propagating $29$-cycle when $b \leq 10000$. The first basis $b$ to have a $29$-cycle is $b=87$, and there are $46$ cycles of length $29$ among the bases $b \leq 1000$.
\if false
For b= 87 number length gt 2 cycles is 5
[ 55, 37, 15, 29, 8 ]
\fi
\end{remark}
\begin{remark} Let $ 2 \leq  \ell_1\leq \dots\leq \ell_s$ be positive integers. One may wonder whether, given such sequence $(\ell_1, \dots, \ell_s)$, there exists 
a basis $b \geq 2$ such that $S_{x^2,b}$ has (at least) $s$ distinct cycles of lengths $\ell_1,\dots, \ell_s$ respectively.  One may also wonder whether, given such sequence $(\ell_1, \dots, \ell_s)$, there exists 
a basis $b \geq 2$ such that $S_{x^2,b}$ has $s$ distinct propagating cycles of lengths $\ell_1,\dots, \ell_s$ respectively. When such is the case,
Theorem \ref{thm.mainpropagating} can be used to produce an arithmetic progression of bases $b$ with that same cycle structure. 
\end{remark}
\fi 
\end{section}

\begin{section}{$2$-cycles for $S_{x^2,b}$}
In addition to the line described in Theorem \ref{thm.mainpropagating}, both varieties $V_1$ and $V_2$ contain a second integer line through each integer point.
We prove this fact in this section for the variety $V_2$ and exploit the existence of this second line to study 
the $2$-cycles of $S_{x^2,b}$.  

\if false Recall  that $V_2$ is the variety of dimension $3$ defined in ${\mathbb A}^5$ 
with coordinates $(b,x,y,u,v)$ by the pair of equations
$$ x^2+y^2=u+bv \quad \text{ and }   \quad u^2+v^2=x+by,$$
(see \ref{emp.equationVell}, up to a change of names for the variables).
 \fi
\begin{proposition}  
\label{pro.twolines}
Let $P:=(b_0,x_0,y_0,u_0,v_0)$  be any point with non-negative rational coefficients on the threefold $V_2$ (defined just before Theorem {\rm \ref{thm.mainpropagating}}) outside of the lines $(t,0,0,0,0)$ and $(t,1,0,1,0)$. Then 
there exist two lines in ${\mathbb A}^5$, defined by equations with coefficients in ${\mathbb Q}$, which are entirely contained in $V_2$ and pass through $P$.
\if false 
The first line is given by 
$$\begin{array}{ccc}
b(t):=b_0+(b_0^2+1)t, &  x(t):=x_0+(x_0b_0-y_0)t,  &   y(t):=y_0+(x_0+y_0b_0)t, \\
&      u(t):=u_0+(u_0b_0-v_0)t,  &   v(t):=v_0+(u_0+v_0b_0)t.
\end{array}$$ 
\if false
$$ 
 d:= 4 y_0^2 u_0^2  + 4 y_0^2 v_0^2 + 4 y_0 v_0 +   v_0^6 + 1 +u_0^3(u_0^3 + 3 u_0 v_0^2 - 2) + u_0 v_0^2(3 u_0 v_0^2- 2 ),  
$$
 
\begin{equation*}
\begin{split}
x_1:=x_0 y_0 u_0^4 + 2 x_0 y_0 u_0^2 v_0^2 - x_0 y_0 u_0 + x_0 y_0 v_0^4 - 2 y_0^3 u_0^2 -
    2 y_0^3 v_0^2 - 3 y_0^2 v_0 + y_0 u_0^3 + y_0 u_0 v_0^2 - y_0 \\
    + u_0^4 v_0 + 2 u_0^2 v_0^3
    - u_0 v_0 + v_0^5,
\end{split}
\end{equation*}
 
$$
y_1:=2 x_0 y_0^2 u_0^2 + 2 x_0 y_0^2 v_0^2 + x_0 y_0 v_0 + y_0^2 u_0^4 + 2 y_0^2 u_0^2 v_0^2 +
    y_0^2 u_0 + y_0^2 v_0^4 + 3 y_0 u_0^2 v_0 + 3 y_0 v_0^3 + v_0^2, 
$$
 
\begin{equation*}
\begin{split}
u_1:= -2 x_0 y_0^2 v_0 + x_0 y_0 u_0^3 + x_0 y_0 u_0 v_0^2 - x_0 y_0 + 2 y_0^3 u_0 - y_0^2 u_0^2 v_0 -
    y_0^2 v_0^3 + y_0 u_0^5 + 2 y_0 u_0^3 v_0^2\\ - y_0 u_0^2 + y_0 u_0 v_0^4  
    - 3 y_0 v_0^2 +
    u_0^3 v_0 + u_0 v_0^3 - v_0, 
\end{split}
\end{equation*}
 
\begin{equation*}
\begin{split}
v_1:=2 x_0 y_0^2 u_0 + x_0 y_0 u_0^2 v_0 + x_0 y_0 v_0^3 + 2 y_0^3 v_0  
+ y_0^2 u_0^3 + y_0^2 u_0 v_0^2 
    + y_0^2 + y_0 u_0^4 v_0  + 2 y_0 u_0^2 v_0^3 \\+ 2 y_0 u_0 v_0   + y_0 v_0^5 + u_0^2 v_0^2  
    +
    v_0^4. 
\end{split}
\end{equation*}
 \fi
The second line is given by 
$$\begin{array}{ccc}
b(t):=b_0+t &    x(t):=x_0+tX_1/D,  &  y(t):=y_0+tY_1/D, \\
&    u(t):=u_0+tU_1/D,  &   v(t):=v_0+tV_1/D,
\end{array}
$$ 
where
\begin{equation*}
\begin{split}
D:=4  y_0^2  u_0^4 + 8  y_0^2  u_0^3 + 8  y_0^2  u_0^2  v_0^2 + 4  y_0^2  u_0^2 + 8  y_0^2  u_0  v_0^2
+
    4  y_0^2  v_0^4 + 4  y_0^2  v_0^2 + 4  y_0  u_0^4  v_0 + 8  y_0  u_0^2  v_0^3 + 4  y_0  v_0^5 +  u_0^6\\
    + 2  u_0^5 + 3  u_0^4  v_0^2 + 3  u_0^4 + 4  u_0^3  v_0^2 + 2  u_0^3 + 3  u_0^2  v_0^4 +
    6  u_0^2  v_0^2 +  u_0^2 + 2  u_0  v_0^4 + 2  u_0  v_0^2 +  v_0^6 + 3  v_0^4 +  v_0^2,
\end{split}
\end{equation*}
\if false
D factors:
 <u0^2 + v0^2, 1>,
    <4*y0^2*u0^2 + 8*y0^2*u0 + 4*y0^2*v0^2 + 4*y0^2 + 4*y0*u0^2*v0 + 4*y0*v0^3 +
        u0^4 + 2*u0^3 + 2*u0^2*v0^2 + 3*u0^2 + 2*u0*v0^2 + 2*u0 + v0^4 + 3*v0^2
        + 1, 1>
        \fi
\begin{equation*}
\begin{split}
X_1:=4  x_0  y_0^3  u_0^2 + 4  x_0  y_0^3  u_0 + 4  x_0  y_0^3  v_0^2 + 4  x_0  y_0^2  u_0^2  v_0 +
    4  x_0  y_0^2  v_0^3 +  x_0  y_0  u_0^4 + 2  x_0  y_0  u_0^3   + 2  x_0  y_0  u_0^2  v_0^2  
   \\ + 2  x_0  y_0  u_0^2 + 2  x_0  y_0  u_0  v_0^2 +  x_0  y_0  u_0 +  x_0  y_0  v_0^4 + 2  x_0  y_0  v_0^2 -
    4  y_0^4  v_0 - 2  y_0^3  u_0^2 - 2  y_0^3  v_0^2 + 2  y_0^2  u_0^2  v_0\\ + 4  y_0^2  u_0  v_0   
    + 2  y_0^2  v_0^3 -  y_0^2  v_0 -  y_0  u_0^4 -  y_0  u_0^3 + 2  y_0  u_0^2  v_0^2 -  y_0  u_0^2 -
     y_0  u_0  v_0^2 + 3  y_0  v_0^4 -  y_0  v_0^2 +  u_0^4  v_0 + 2  u_0^3  v_0\\ + 2  u_0^2  v_0^3  
     +2  u_0^2  v_0 + 2  u_0  v_0^3 +  u_0  v_0 +  v_0^5 + 2  v_0^3,
\end{split}
\end{equation*}

\begin{equation*}
\begin{split}
Y_1:=4  x_0  y_0^3  v_0 + 2  x_0  y_0^2  u_0^2 + 2  x_0  y_0^2  v_0^2 +  x_0  y_0  v_0 + 4  y_0^4  u_0^2 +
    4  y_0^4  u_0 + 4  y_0^4  v_0^2 + 4  y_0^3  u_0^2  v_0 + 4  y_0^3  v_0^3 +  y_0^2  u_0^4 \\+
    2  y_0^2  u_0^2  v_0^2 +  y_0^2  u_0 +  y_0^2  v_0^4 + 4  y_0^2  v_0^2 + 3  y_0  u_0^2  v_0 +
    3  y_0  v_0^3 +  v_0^2,
\end{split}
\end{equation*}

 \begin{equation*}
\begin{split}
U_1:= 4  x_0  y_0^2  u_0  v_0 + 2  x_0  y_0^2  v_0 +  x_0  y_0  u_0^3 +  x_0  y_0  u_0^2 +  x_0  y_0  u_0  v_0^2 +
     x_0  y_0  u_0 -  x_0  y_0  v_0^2 + 4  y_0^3  u_0^3 + 6  y_0^3  u_0^2 \\
     + 4  y_0^3  u_0  v_0^2 + 2  y_0^3  u_0 + 2  y_0^3  v_0^2 + 4  y_0^2  u_0^3  v_0 -  y_0^2  u_0^2  v_0 + 4  y_0^2  u_0  v_0^3 -
    2  y_0^2  u_0  v_0 -  y_0^2  v_0^3 -  y_0^2  v_0 +  y_0  u_0^5 \\
    +  y_0  u_0^4 + 2  y_0  u_0^3  v_0^2  +y_0  u_0^3 +  y_0  u_0  v_0^4 + 5  y_0  u_0  v_0^2 -  y_0  v_0^4 + 2  y_0  v_0^2 +  u_0^3  v_0 +
     u_0^2  v_0 +  u_0  v_0^3 +  u_0  v_0 -  v_0^3,
\end{split}
\end{equation*}
 \begin{equation*}
\begin{split}
 V_1:=-2  x_0  y_0^2  u_0^2 - 2  x_0  y_0^2  u_0 + 2  x_0  y_0^2  v_0^2 +  x_0  y_0  u_0^2  v_0 + 2  x_0  y_0  u_0  v_0
    +  x_0  y_0  v_0^3 +  x_0  y_0  v_0 + 4  y_0^3  u_0^2  v_0 + 4  y_0^3  u_0  v_0 \\
    + 4  y_0^3  v_0^3 +2  y_0^3  v_0 +  y_0^2  u_0^3 + 4  y_0^2  u_0^2  v_0^2 + 3  y_0^2  u_0^2 +  y_0^2  u_0  v_0^2 +
     y_0^2  u_0 + 4  y_0^2  v_0^4 +  y_0^2  v_0^2 +  y_0  u_0^4  v_0 + 2  y_0  u_0^3  v_0 \\
     + 2  y_0  u_0^2  v_0^3 -  y_0  u_0^2  v_0 + 2  y_0  u_0  v_0^3 - 2  y_0  u_0  v_0 +  y_0  v_0^5 +3  y_0  v_0^3 +  u_0^2  v_0^2 + 2  u_0  v_0^2 +  v_0^4 +  v_0^2.
\end{split}
\end{equation*}
In particular, the coefficients in the formula for 
$D$ and $Y_1$ are all positive, and $D>0$ at $P_0$ since $u_0$ and $v_0$ are not both zero by hypothesis.
 \fi
\end {proposition}

\if false
To verify directly that the two parametric lines given in the statement of the proposition are on $V_2$, one may proceed as follows.
Since the process is the same for both lines, let us describe it only for the second line. Consider the expressions $D,X_1,Y_1,U_1,V_1$ as polynomials  in  
${\mathbb Q}[x_0,y_0,u_0,v_0,b_0]$. Set $x:=Dx_0+tX_1$, $y:=Dy_0+tY_1$, $u:=Du_0+tU_1$, $v:=Dv_0+tV_1$, and $b:=Db_0+Dt$, viewed as polynomials in 
${\mathbb Q}[x_0,y_0,u_0,v_0,b_0][t]$.  Evaluate $x^2+y^2-(Du+bv)$ and  $u^2+v^2-(Dx+by)$ to obtain two polynomials in $t$. 
Check that each coefficient $c$ of each polynomial is in the ideal $I=(f_0,g_0)$. This can be done 
using Magma \cite{Magma} by defining $I:=\texttt{IdealWithFixedBasis}([f_0,g_0])$ and asking for $\text{Coordinates}(I, c)$.

The expression for $D$ is a sum of non-negative terms which include $u_0^2+v_0^2$. Thus $D>0$ at $P_0$ since $u_0$ and $v_0$ are not both zero by hypothesis.
Note that in fact $u_0^2+v_0^2$ divides the expression $D$. Let us explain here how the equations of these lines were found. 
\fi

\begin{proof}
Start with ten variables $x_0, x_1,y_0,y_1,u_0,u_1,v_0,v_1,$ and $b_0,b_1$. Evaluate the two equations  for $V_2$ at  
the linear polynomials $b(t):=b_0+tx_1$, $x(t):=x_0+tx_1$, $y(t):=y_0+ty_1$, $u(t):=u_0+tu_1$, and $v(t):=v_0+tv_1 $ 
to obtain two quadratic polynomials in $t$, say $f:=f_2t^2+f_1t+f_0$ and $g:=g_2t^2+g_1t+g_0$. Forcing these two polynomials to vanish identically produces six equations in the ten
variables. The constant terms $f_0=x_0^2+y_0^2-(u_0+b_0v_0)$ and $g_0=u_0^2+v_0^2-(x_0+b_0y_0)$ 
are just the two equations of $V_2$ evaluated at the $0$-variables.  

Magma \cite{Magma} can verify that $(f_0,g_0)$ is a prime ideal in ${\mathbb Q}[x_0,y_0,u_0,v_0,b_0]$. Let $F$ denote the field of fractions of the ring ${\mathbb Q}[x_0,y_0,u_0,v_0,b_0]/(f_0,g_0)$.
Working now in the polynomial ring $F[x_1,y_1,u_1,v_1,b_1]$, consider the ideal $I:=(f_1,f_2,g_1,g_2)$.
Use the Magma \cite{Magma} function $\texttt{PrimaryDecomposition}(I) $ to produce the primary decomposition of this ideal. 
After about $34$ hours of computing time, Magma will produce a decomposition which consists of three distinct ideals. 
Two of these ideals have generators that can be used to produce the parametric formulas for two different lines defined over ${\mathbb Q}$.
Both lines  at this point have parametric equations which are too long and complicated to be printed in this article.  
We succeeded in simplifying the parametrization of one of the lines, and checked that this line is the same line as the line exhibited in Theorem \ref{thm.mainpropagating}. 

For the remainder of this section, let us call {\it the second line} through a point $P$ on $V_2$  the line whose existence is established in the   proposition and which is not equal to the line exhibited  in Theorem \ref{thm.mainpropagating}.  The Magma computation allows us to give this line 
in parametric form $$\begin{array}{ccc}
b(t):=b_0+t &    x(t):=x_0+tX_1/D,  &  y(t):=y_0+tY_1/D, \\
&    u(t):=u_0+tU_1/D,  &   v(t):=v_0+tV_1/D,
\end{array}
$$ 
where the coefficients $D$, $X_1$, $Y_1$, $U_1$, $V_1$
are long formulas in the variables $b_0,x_0,y_0,u_0,v_0$. 
For instance,
\begin{equation*}
\begin{split}
D:= u_0^2 +  v_0^2 + 4  y_0^2  u_0^4 + 8  y_0^2  u_0^3 + 8  y_0^2  u_0^2  v_0^2 + 4  y_0^2  u_0^2 + 8  y_0^2  u_0  v_0^2
+ 4  y_0^2  v_0^4 + 4  y_0^2  v_0^2 \\ + 4  y_0  u_0^4  v_0  
+ 8  y_0  u_0^2  v_0^3 + 4  y_0  v_0^5 +  u_0^6 
    + 2  u_0^5 + 3  u_0^4  v_0^2 + 3  u_0^4 + 4  u_0^3  v_0^2 + 2  u_0^3\\  + 3  u_0^2  v_0^4 +
    6  u_0^2  v_0^2 
    + 2  u_0  v_0^4 + 2  u_0  v_0^2 +  v_0^6 + 3  v_0^4,
\end{split}
\end{equation*}
and we see that since $D$ is a sum of  monomials which includes $u_0^2+v_0^2$, 
we must have $D>0$ at $ P$ since the point $P$  has non-negative coefficients and since $u_0$ and $v_0$ are not both zero by hypothesis.
\end{proof}

\begin{remark}
The line exhibited  in Theorem \ref{thm.mainpropagating} is remarkable since it allows us to `propagate' any given propagating cycle. 
We believe that the second line on $V_2$ has the same property. In particular, starting with a propagating $2$-cycle, 
we expect that the expressions $X_1,Y_1,U_1,V_1$ are all non-negative. This is immediately true for $Y_1$ since Magma produces a formula which is a sum of monomials, but 
for the other expressions, the formula involves some negative signs.
The propagating property on the other hand can always be checked directly given an explicit propagating cycle, and this is what we do in order to establish our next proposition. 
\end{remark}

Assume that the point $P:=(b_0,x_0,y_0,u_0,v_0)$ on $V_2$ has integer coefficients. Then the second lines can be reparametrized using a change of variables of the form $t:=\lambda s$ with $\lambda \in {\mathbb N}$, 
so that the new equations for the lines 
have only integer coefficients. 
In general, there are very large cancellations
in the fractions 
$X_1/D, Y_1/D,U_1/D$ and $V_1/D$, and we set in this case $\lambda$ to be the least common multiple of the denominators  of 
$X_1/D, Y_1/D,U_1/D$ and $V_1/D$.
We can use the  second line through  propagating $2$-cycles to  improve the lower bound given by Proposition \ref{pro.3-4percent}.

\begin{proposition} \label{pro.percent} The set $\operatorname{PB}(2)$ of bases $b \geq 2$ such that $S_{x^2,b}$ has a propagating $2$-cycle has lower density bounded below by $0.5763$.
\end{proposition}
\begin{proof} Consider the set $P_N$ of all propagating $2$-cycles with $2 \leq b \leq N$.  
As in Proposition \ref{pro.3-4percent}, for each propagating cycle ${\rm cyc}(n_1,n_2)$ in base $b$ in $P_N$, 
Theorem \ref{thm.mainpropagating} produces an integer line, 
which we parametrize using its reduced integer parametrization (see Remark \ref{rem.reduced}). 
The slope of the line is $\lambda_1:=(b^2+1)/\gcd(b^2+1,n_1,n_2)$. Consider the set $S_1$ of all the lines found this way.

Now for each propagating $2$-cycle in base $b$ in $P_N$, say ${\rm cyc}(x_0+by_0, u_0+bv_0)$,
 compute the reduced integer parametrization of the second line, with $b(t)=b+\lambda_2t$,
$x(t)=x_0+x_2t$, $y(t)=x_0+y_2t$, $ u(t)=u_0+u_2t$, and $v(t)=v_0+v_2t$, 
and check that this line allows us to propagate the $2$-cycle. To check this, we verified that $0\leq x_2,y_2,u_2,v_2 \leq \lambda_2$. 
Consider the set $S_2$ of all the second lines found this way.
We now have a set of explicit arithmetic progressions
of the form $b+\lambda_1 t$ or $b+\lambda_2 t$ which are contained in $\operatorname{PB}(2)$.

To compute a lower bound for the lower density, we further prune the set  $S_1 \cup S_2$ from all the lines whose slope is not 
 a power of a prime. 
 The lower density of the set of arithmetic progressions associated with the remaining lines can be easily bounded below since the lines have slopes that are all pairwise coprime
 and we can use the formula in Remark \ref{rem.arithprog} to obtain a lower bound of 
 $0.5457$ when $N=1000$ and $|P_N|=2885$. 
 We can do slightly better by also considering some lines whose slope is not a power of a prime.
 For instance,  when also considering the lines with slopes dividing $2 \cdot 17^3$, we obtain a lower bound of 
 $0.5763$ when $N=1000$.
\end{proof}

Let $P=(b_0,x_0,y_0,u_0,v_0)$ be a propagating cycle on $V_2$, and consider the two lines passing through it and their reduced integer parametrization
with $b_1(t)=b_0+\lambda_1t$ and $b_2(t)=b_0+\lambda_2t$. We have noted already that $\lambda_1=(b_0^2+1)/\gcd(b_0^2+1,n_1,n_2)>1$, so that for any positive integer $\mathbf{t}$, $b_1^2(\mathbf{t})+1$ is never prime. Thus no propagated cycle on the first line can have a base $b$ such that $b^2+1$ is prime. 
On the other hand, quite often, the second line can produce   propagated cycles that have a base $b$ such that $b^2+1$ is prime. Our next proposition exploits this property.

\begin{proposition} \label{thm.HPB}
\begin{enumerate}[\rm (a)]
\item  There exist infinitely many integers $b\geq 2 $ such that $S_{x^2,b}$ has exactly two non-trivial $1$-cycles, but 
  $ b \notin L(x^2,3)$ because $S_{x^2,b}$ also has a $2$-cycle.
\item  The Bouniakovsky Conjecture implies that there exist infinitely many integers $b\geq 2 $ such that $S_{x^2,b}$ has no non-trivial $1$-cycles,
but $b \notin L(x^2,2)$ because $S_{x^2,b}$ also has two distinct $2$-cycles.
\end{enumerate}
\end{proposition}
\begin{proof} Part (a) follows  from the existence of the second line on $V_2$  passing through the point 
$P:=(8,2,3,5,1)$, and  given by 
$$ \quad b=17t + 8,
x=3t + 2,
y=5t + 3,
u=9t + 5,
v=2t + 1.
$$
Indeed, this line has $b(t)=17t+8$, with $\gcd(17, 65)=1$. It follows from Lemma \ref{lem.arithprog} that the integer values of $(17x+8)^2+1$ are coprime
and so we can  use \cite[Theorem, p. 172]{Iwa}, fully proved in \cite[Theorem 1]{Lem}, applied to the polynomial $(17x+8)^2+1$ to obtain that there are infinitely many values $b$ in the arithmetic progression 
$b(t)=17t+8$ such that $b^2+1$ is the product of two primes. It is clear from the equation of the line that for every positive $t$, $S_{x^2,b(t)}$ has a propagating $2$-cycle.

Part (b) follows from the existence of a second line with a similar property. Starting with the $2$-cycle $(24 ,16,6,4,12)$, we find 
using the proof of Proposition \ref{pro.twolines}  that the second line on $V_2$ through that point is given by
$$  \quad b=53t + 24,
x=34t + 16,
y=13t + 6,
u=8t + 4,
v=25t + 12.
$$
Again, all the coefficients of the line are positive, and it is easy to verify that for all $t>0$, $x,y,u,v<b$. 
Thus for each $t$, the corresponding $b$ is such that $S_{x^2,b}$ has a $2$-cycle. It is easy to verify that $\gcd(53, 24^2+1)=1$. 
Finally, we can find a point in the intersection of the arithmetic progressions $17t+8$ and $53t+24$; for instance when $x_0=400$ and $x_1=128$, 
we have $17x_0+8=53x_1+24=6808 $. Thus we can consider the progression $17 \cdot 53 t + 6808$, and Lemma \ref{lem.arithprog} (b) shows that the integer values of 
$(17\cdot 53x+6808)^2+1$ are coprime. Hence, the Bouniakovsky Conjecture implies that there exist infinitely many integers $t$ such that 
$(17\cdot 53t+6808)^2+1$ is prime and, therefore, there exist infinitely many integers $b$ of the form $b=17\cdot 53t+6808$ such that $b^2+1$ is prime. For each such integer, 
we find that $S_{x^2,b}$ has two $2$-cycles by construction.
\end{proof}

\begin{lemma} \label{lem.arithprog}
\begin{enumerate}[\rm (a)]
\item
Let $c,d \in {\mathbb Z}$. The values of the polynomial $(cx+d)^2+1$ when $x$ is an integer are coprime if and only if $\gcd(c, d^2+1)=1$. 
\item Let $c_0,d_0,c_1,d_1 \in {\mathbb Z}$. Suppose that  $\gcd(c_0, d_0^2+1)=1$ and $\gcd(c_1, d_1^2+1)=1$. 
Suppose that there exist integers $x_0$ and $x_1$ such that $c_0x_0+d_0= c_1x_1+d_1$. 
Then the integer values of the polynomial $(c_0c_1x+ c_0x_0+d_0)^2+1$ are coprime.
\end{enumerate}
\end{lemma}
\begin{proof} (a) If $p$ is a prime which divides all the values of $(cx+d)^2+1$ when $x$ is an integer, then $p$ divides $d^2+1$, $c(c+2d)$ and $c(c-2d)$. 
Hence, if $p$ does not divide $c$, then $p$ divides $c+2d$ and $c-2d$, and thus divides $4d$. It follows that $p=2$. But this is a contradiction, 
since then $c$ is odd, and then $c^2+2cd$ is also odd. Thus $p$ divides $\gcd(c, d^2+1)$. 
Reciprocally, if $p$ divides $c$ and $d^2+1$, it is clear that $p$ divides all the values of $(cx+d)^2+1$ when $x$ is an integer.

(b) In view of part (a), it suffices to prove that $\gcd((c_0x_0+d_0)^2+1, c_0c_1)=1$.

\end{proof}

\begin{remark}  \label{rem.prop}
The two lines used in the proof of Proposition \ref{thm.HPB} are far from unique, and many other such lines could have been used.
In fact, we believe that there are infinitely many propagating $2$-cycles $P$ on $V_2$ such that the second line passing through $P$
in reduced integer parametrization
with $b_2(t)=b_0+\lambda_2t$ is such that $\gcd(b_0^2+1,\lambda_2)=1$.

\if false
described in Proposition \ref{pro.twolines}
has the properties (a) and (b) below. Let $P:=(b_0, x_0,y_0,u_0,v_0)$ denote a $2$-cycle, and reparametrize as in \ref{rem.reparameterize}
the second line to obtain an reduced integer parametrization such that $b(t):=b_0+\lambda_2 t$ with $\lambda_2$ the least common multiple of the denominators of 
$X_1/D, Y_1/D,U_1/D$ and $V_1/D$. Then
\begin{enumerate}[\rm (a)]
\item All coefficients $X_1/D, Y_1/D,U_1/D$ and $V_1/D$ are non-negative, so that for each integer $t \geq 0$, the integer point on the line corresponds to a propagating $2$-cycle.
\item $\gcd(\lambda_2, b_0^2+1)=1$. 
\end{enumerate}

Let $n:=x_0+b_0y_0$ and $m:=u_0+b_0v_0$. Let $g:=\gcd(n,m, b_0^2+1)$.
Computations suggest that in general, $\gcd(\lambda_2, b_0^2+1)$ divides $g$. 
It seems likely that there  exist infinitely many propagating $2$-cycles with $\gcd(n,m)=1$ or $b_0^2+1$ prime, and if this is the case, 
one should then expect infinitely many instances where $\gcd(\lambda_2, b_0^2+1)=1$. 

Note that the set of possible factorizations of $g$ is restricted by the fact that $b_0^2+1$, $n$, and $m$, are sums of squares.
It can be shown directly that if the line has integer parametrization $b(t)=b_0+\lambda_2t$, $x(t)=x_0+x_1t$, $y(t)=y_0+y_1t, \dots$, then $\lambda_2$ divides $x_1^2+y_1^2$ and thus the factorization of $\lambda_2$ is restricted in a similar way.

\if false
Note that the first line through $P$, when written
in reduced integer form with $b(t):=b_0+\lambda_1 t$, is likely to have $\lambda_1=(b_0^2+1)/g$. Indeed, we mention in \ref{rem.reparameterize} that 
the first line is likely to be equal to the line described in Theorem \ref{thm.mainpropagating}, 
and the reduced integer parametrization of that line is described in \ref{rem.reduced}.
\fi
\end{remark}

\begin{remark} 
\fi 
\label{rem.data1}
To speed up the verification of Theorem \ref{thm.data0}, we created in advance a set of about $150$ different second lines in reduced integer parametrization  
with $\gcd(b_0^2+1,\lambda_2)=1$,
and computed all bases $b \leq 10^6$ such that $b^2+1$ is prime and such that $S_{x^2,b}$ has a known $2$-cycle on one of our $150$ such lines. To eliminate such a $b$ from $L(x^2,2)$ required then to produce only one cycle of length greater than $2$ for $S_{x^2,b}$, which is a very quick computation. 
\end{remark}
\begin{example}
Consider $b_0:=288504$, an unusual base discovered when verifying Theorem \ref{rem.data0extended}. 
In this case, $S_{x^2,b_0}$ has exactly $104$ distinct cycles, all of them of length at most $7$. All non-trivial cycles are propagating, 
with forty-seven $1$-cycles, thirty-nine $2$-cycles, ten $3$-cycles, six $5$-cycles, and one cycle of length $4,6$, and $7$, respectively.

To illustrate propagation, note that the very first base $b>2$ to have a $7$-cycle is $b=15$, with $c:={\rm cyc}(50,34,20,26,122,68,80)$ (see \cite{H-P}, page 10).\footnote{Note a typo
in  \cite{H-P} in the list of cycles for $b=15$: the cycle just before  $c$ should 
be the  non-propagating $5$-cycle ${\rm cyc}(41, 125, 89, 221, 317)$.}
The reduced integer parametrization of the line passing through this $7$-cycle starts with $b(t)=113t + 15$, $
x_1(t)=36t + 5$, $y_1(t)=25t + 3$, etc. We find that when ${\bf t}=2553$, the corresponding integer point on the line produces the $7$-cycle for $b({\bf t})=288504$ found in our search.  
 
All cycles of length at least $2$ are found among the orbits of $n$ with $1\leq n \leq 1,964,329,269$.
Among the thirty-nine propagating $2$-cycles, one of them, ${\rm cyc}(36253850477, 38091031810)$ is such that the second line associated 
to it has $\gcd(b_0^2+1,\lambda_2)=1$ when written in reduced integer parametrization.

\end{example}

\if false
\begin{remark}
Consider the variety $W_2$ in ${\mathbb A}^5$ defined by the pair of equations
$ x^2+y^2=u+bv +b^2 $ and $u^2+v^2+1=x+by$. One can associate an integer point  on $W_2$ to any $2$-cycle
which is not special. When computing $2$-cycles explicitly for many bases $b_0$, one quickly notes that there seems to be a lot more special $2$-cycles than of $2$-cycles that are not special. This phenomenon can be explained by the fact that the arithmetic geometry of the varieties $V_2$ and $W_2$ are different. The variety $V_2$ contains a web of integer lines, while the variety $W_2$ does not.
\end{remark}
\fi

\if false
\begin{remark} 
Let $b \leq 200,000$. If $b \in L(x^2,3)$, then 
$$b\in  \{ 14, 66, 94,   300,   384, 436, 496, 750, 1406, 1794, 
   2336,   2624, 28034 \}.$$
\end{remark}
\fi

\if false
\begin{remark} \label{rem.data}
We complement the computations presented in Theorem \ref{thm.data0} with the following data. 
\begin{enumerate}
\item Let $b \leq 200,000$. If $b \in L(x^2,3)$, then 
$$b\in L_3:=\{ 14, 66, 94,   300,   384, 436, 496, 750, 1406, 1794, 
   2336,   2624, 28034 \}.$$

\item Let $b \leq 100,000$. If $b \in L(x^2,2,1)$ (i.e., $b$ has exactly two cycles of length greater than $1$), then either $b \in L_3$, or
$$b\in \{
7, 9, 11, 17, 22, 30, 44, 234, 292, 330, 342, 370, 4994, 10848, 13750, 21142, 57256 \}.$$
\item Let $b \leq 6000$. If $b \in L(x^2,0,2)$, then $b \in \{2,3,4,13,18,92\}$. In other words, $b=2,3,4,13,18,92$ are the only known bases
where $S_{x^2,b}$ has only cycles of length at most $2$.  
\item Let $b \leq 10,000$. Then $b \notin L(x^3,3)$.
\end{enumerate}
\end{remark}
\fi

\end{section}

\begin{section}{$1$-cycles of $S_{x^2,b}$}
 
 In this section, we  complement Subramanian's Theorem \ref{thm.Sub} with Proposition \ref{pro.1cyle}, and further study the  surface $V_1$ associated 
 with $1$-cycles when  $\phi(x)=x^2$.
 
 \begin{proposition}  \label{pro.1cyle}
Let $n:=x+by$ be a non-trivial $1$-cycle for $S_{x^2,b}$, and let $d:=\gcd(b^2+1,n)$. Then $d>1$, 
and there exists another $1$-cycle $N:=x+b(b-y)$ for $S_{x^2,b}$ such that, letting  $D:=\gcd(b^2+1,N)$, we have $D>1$ and  $b^2+1=dD$.

Let now $g:=\gcd(x,y)$, $g':=\gcd(x,b-y)$, $h:=\gcd(x-1,y)$, and $h':=\gcd(x-1,b-y)$.  Then 
\begin{enumerate}[\rm (a)]
\item $x=gg'$, $y=gh$, $b-y=g'h'$, and $x-1=hh'$.
\item $d=n/g^2$, 
and $D=N/g'^2$. 

\end{enumerate} 
\end{proposition}
 \begin{proof} 
  We leave it to the reader to check that $N$ is a non-trivial $1$-cycle. 
 By hypothesis, $x(x-1)=y(b-y)$. Since $x$ and $x-1$ are coprime, we find that $y=gh$ and $b-y=g'h'$. Then 
 $x(x-1)=(gg')(hh')$. By definition, again since $x$ and $x-1$ are coprime, we find that $gg'$ is coprime to $x-1$, and it follows that $gg'$ divides $x$. 
 The same argument shows that $hh'$ divides $x-1$. The equality $x(x-1)=(gg')(hh')$ implies then that $gg'=x$ and 
 $hh'=x-1$, proving Part (a).

 We claim that $d>1$. Indeed, if $d=1$, then $n=x+by$ divides $x+by-y^2$ and, hence, divides $y^2$.
 This is a contradiction since $0<y^2< x+by$. Similarly,  $D>1$ since otherwise $N$ divides $x+by-y^2$, which is also a contradiction since 
 $N> x+by-y^2$. The reader will check directly that
 $$ nN= (b^2+1)(x+by-y^2)=(b^2+1)x^2 .$$ 
 Note that it follows from this equality that $n \neq N$, since otherwise $b^2+1$ would be a square, which is not possible since $b>0$.
 Using Part (a) and this equality, 
 we find that $$
 (n/g^2)(N/g'^2)=b^2+1.$$
 It is then clear from this latter equality that $(n/g^2)$ divides $d$, and that $(N/g'^2)$ divides $D$. 
 To finish the proof of Part (b), it suffices to prove that $dD=b^2+1$.  
 

 For this, it suffices to show that for every prime $p$, we have $\text{ord}_p(dD) = \text{ord}_p(b^2+1)$.
 First, note that $dD$ and $b^2+1$ have the same prime divisors. Indeed, it is clear from the definitions that if $p$ divides either $d $ or $D$, then it divides $b^2+1$. On the other hand, if $p$ divides $b^2+1$, then it divides $nN$ and, hence, it divides $d$ or $D$.
 
 Let us show now that  for every prime $p$, $\text{ord}_p(dD) \leq \text{ord}_p(b^2+1)$. The inequality is clear if either $\alpha:=\text{ord}_p(d)=0$ or 
$ \beta:=\text{ord}_p(D)=0$, so we may assume that $\alpha,\beta >0$. Then $p^\alpha \cdot p^\beta$ divides $nN$ and since $nN=(b^2+1)x^2$, we obtained the desired inequality if we show that $p$ does not divide $x$.  Assume by contradiction that $p $ divides $x$. Since by hypothesis, $p$ also divides $n=x+by$, 
we find that $p$ divides $by$ and, hence, $p$ divides $y$ because  $p $ cannot divide $ b$ since it divides $b^2+1$. 
Again by hypothesis, $p$ divides $N$, so $p $ divides  $N-x+by=b^2$, which is impossible. As a result, $p $ does not divide $ x$.

We now show that for every prime $p$, $\text{ord}_p(dD) \geq \text{ord}_p(b^2+1)$. Let $\gamma:=\text{ord}_p(b^2+1)>0$,
and assume by contradiction that $\gamma>\alpha+\beta$. Then $\gamma>\alpha$ and $\gamma>\beta$, which by definition implies that $\alpha= \text{ord}_p(n)$ and 
$\beta=\text{ord}_p(N)$. This is a contradiction, since we can conclude from 
 $(b^2+1)x^2= nN$  that $\gamma \leq  \text{ord}_p(nN)=\alpha+\beta$.  
 \if false
 Part (a) is proved at the end of the proof of Proposition \ref{pro.1cyclelines}.
 A direct computation shows that  $ n((x-1)^2+y^2)= (b^2+1)y^2-(x^2+y^2-x-by)$, and thus 
 $$ n((x-1)^2+y^2)= (b^2+1)y^2.$$
Using this formula, we obtain
$$ \frac{n}{d} \frac{(x-1)^2+y^2}{h^2}= \frac{(b^2+1)}{d}\frac{y^2}{h^2}.$$
Using  that $b^2+1=Dd$, we find that $n=d(y/h)^2$ if and only if $D=((x-1)^2+y^2)/h^2$, thus proving part (b).

A direct computation shows that 
$$\left( (\frac{x-1}{h'})^2+(\frac{b-y}{h'})^2\right) \left( (\frac{x-1}{h})^2+(\frac{y}{h})^2\right)= (b^2+1) \left(\frac{x-1}{hh'}\right)^2.$$
Since the left hand side of this equation is equal to $dD$ and we proved that $dD=b^2+1$, we find that Part (c) follows.

Since $x(x-1)=y(b-y)$, we find that $y=\gcd(x(x-1),y)=gh$, since $x$ and $x-1$ are coprime.
\fi
\end{proof}

\begin{proposition} \label{pro.1cyclelines}
Recall 
the surface $V_1$ in  ${\mathbb A}^3$   
given by the equation $x^2+y^2-(x+by)=0$. Given any point $P:=(b_0,x_0,y_0)$ on $V_1$ outside of the lines $(t,0,0)$ and $(t,1,0)$, 
there exist exactly two lines in ${\mathbb A}^3$ that are entirely contained in $V_1$ and pass through $P$, namely the two lines given by
the parametric equations
$$
\begin{array}{lll} b(t):=b_0+t((x_0-1)^2+y_0^2), &   x(t):=x_0+t(x_0-1)y_0,  &  y(t):=y_0+ty_0^2,\\
b(t):=b_0+t(x_0^2+y_0^2), &  x(t):=x_0+tx_0y_0,  &  y(t):=y_0+ty_0^2.
\end{array}
$$
The first parametric equation parametrizes the same line through $P$ as the line given in Theorem {\rm \ref{thm.mainpropagating}}.
Let $d:=\gcd(b_0^2+1,x_0+b_0y_0)$, and $D:=(b_0^2+1)/d$. In reduced integer form, the first line has $b(t)=b_0+Dt$ and the second line   has $b(t)=b_0+dt$. 
\end {proposition}
\begin{proof} It is straightforward to check that the two lines described in the proposition lie on the surface $V_1$.
The equations for these two lines were found using the same method as in the proof of Proposition \ref{pro.twolines}, and this method shows that exactly two lines through $P$ exist. 
\if false
Start with six variables $x_0, x_1,y_0,y_1$ and $b_0,b_1$. 
Evaluate the equation  for $V_1$ at  
the linear polynomials $b(t):=b_0+tx_1$, $x(t):=x_0+tx_1$, and $y(t):=y_0+ty_1$
to obtain a quadratic polynomials in $t$, say $f(t):=f_2t^2+f_1t+f_0$. Forcing this polynomials to vanish identically produces three equations in the five
variables. The constant term $f_0$ of  $f$ 
is the equation of $V_1$ evaluated at the $0$-variables.  Let $F$ denote the field of fractions of the ring ${\mathbb Q}[x_0,y_0,b_0]/(f_0)$.
Working now in the polynomial ring $F[x_1,y_1,b_1]$, consider the ideal $I:=(f_1,f_2)$.
Use the Magma \cite{Magma} function $\texttt{PrimaryDecomposition}(I) $ to produce the primary decomposition of this ideal. It consists of two distinct prime ideals which have generators that can be used to produce the formulas in the proposition.
\fi
Recall that the  line in 
Theorem {\rm \ref{thm.mainpropagating}}  is given by the parametric equations
$$B(t)=b_0+(b_0^2+1)t,\quad    X(t)=x_0+(x_0b_0-y_0)t, \quad   Y(t)=y_0+(x_0+b_0y_0)t.$$
To see  that it equals the first line of the proposition, we make the change of variable $t=(b_0^2+1)s$ in the first line, and $t=((x_0-1)^2+y_0^2)s$ in the other line, so that $b((b_0^2+1)s)=B(((x_0-1)^2+y_0^2)s)$. It remains to note  that 
$$\begin{array}{rcll}
\frac{x((b_0^2+1)s)-x_0}{s} & =& (x_0-1)y_0(b_0^2+1) & \\
& =& (x_0b_0-y_0)((x_0-1)^2+y_0^2) &  =\frac{X(((x_0-1)^2+y_0^2)s) -x_0}{s},
  \end{array}
  $$
  and 
  $$\begin{array}{rcll}
\frac{y((b_0^2+1)s)-y_0}{s} & =& y_0^2(b_0^2+1) & \\
&= & (x_0+b_0y_0)((x_0-1)^2+y_0^2) 
 &=  \frac{Y(((x_0-1)^2+y_0^2)s)-y_0}{s}.
  \end{array}
  $$
\if false  
  Let us now consider the reduced integer form of the first line, for which we have two different parametrizations.
  Using the parametrization with $b(t)=b_0+(((x_0-1)^2+y_0^2)t$, we find that the greatest common divisor of $((x_0-1)^2+y_0^2)$, $(x_0-1)y_0$, and $y_0^2$,
  is $h^2$, with $h:=\gcd(x_0-1,y_0)$. Using the parametrization with $B(t)=b_0+(b_0^2+1)t$, we find that the greatest common divisor of $b_0^2+1$, $x_0b_0-y_0$, and $x_0+b_0y_0$,
  is $d:=\gcd(b_0^2+1,n)$. Hence, we must have $((x_0-1)^2+y_0^2)/h^2= (b_0^2+1)/d$. Since we showed in Proposition \ref{pro.1cyle} that $D=(b_0^2+1)/d$, 
  we have proved one of the two statements in Proposition \ref{pro.1cyle} (a). The second statement follows by applying the above discussion to the lines
  passing through the point corresponding to the $1$-cycle $N=x_0+b(b-y_0)$.
  Let $d:=\gcd(b_0^2+1,x_0+b_0y_0)$, and $D:=(b_0^2+1)/d$. In reduced integer form, the first line in Proposition \ref{pro.1cyclelines} has $b(t)=b_0+Dt$. The second line in reduced integer form has $b(t)=b_0+dt$. 
  \fi
\end{proof}

\end{section}

 \begin{section}{Short cycles of $S_{x^3,b}$}
 
  In this section, $\phi(x)=x^3$. 
 We exhibit below several parametric families of $1$-cycles for $S_{x^3,b}$. After we became aware of \cite{D-J2}, we noted that most of Proposition \ref{pro.cubic1cycles} already appears as Theorems 2-5 in \cite{D-J2}. Only parts (c) and (d) in Proposition \ref{pro.cubic1cycles} are in slightly stronger form than in \cite{D-J2}.

 \begin{proposition} \label{pro.cubic1cycles} Let $k \geq 1 $ be a positive integer.
 
 \begin{enumerate}[\rm (a)]

\item   Let $b=3k+1$. Then $n:=[2k+1,0,k+1]_b$,  $n:=[0, 2k+1,k]_b$, and  $n:=[1, 2k+1,k]_b$ are $1$-cycles for $S_{x^3,b}$. 

\item  Let  $b=3k+2$. Then $n:=[2k+1,0,k]_b$ is a $1$-cycle for $S_{x^3,b}$.

\item Let $ b = 9k+3$.   Then $n:=[6k+2, 4k+2, 5k+1]_b$ is a $1$-cycle for $S_{x^3,b}$.  

\item  Let $ b = 9k+6$.   Then $n:=[6k+4, 2k+1, 7k+5]_b$ is a $1$-cycle for $S_{x^3,b}$.   
\end{enumerate}
\end{proposition}
\begin{proof}
An integer $n:=[x,y,z]_b$ is a $1$-cycle for $S_{x^3,b}$ if and only if the equation
\begin{equation*} \label{eq.cubic}
x^3+y^3+z^3=x+yb+zb^2
\end{equation*}
is satisfied. That this is the case can be checked directly.
\end{proof}

\if false
\begin{remark} It might be worth indicating how some of these parametric families were found.  
The equation \eqref{eq.cubic} defines a $3$-fold in the $4$-dimensional affine space 
with coordinates $(b,x,y,z)$. Each of the parametric solutions in Proposition \ref{pro.cubic1cycles} defines a line on that $3$-fold. The reader will note that these lines in fact lie on specific hyperplane sections of the $3$-fold, namely on $y=0$ or $x=0$ in (a) and (b), and on $y+z=b$ in (c) and (d).
Each such hyperplane section can then be considered as a one-parameter family of cubic plane curves, such as $x^3+z^3=x+b^2z$ or $y^3+z^3=b(y+bz)$ in case $(a)$.
The command PointSearch() in Magma \cite{Magma} is very efficient at finding rational solutions on such curves, and we spotted these parametric solutions after `staring' long enough at a collection of solutions.
\end{remark}
\fi

 \begin{remark} There are bases $b$ of the form $b=9k$, such as $b=72$, $90$, or $270$, for which $S_{x^3,b}$ does not have any non-trivial $1$-cycle.
 The bases $b=18,27$, and $54$ have exactly one non-trivial $1$-cycle, and $b=108$ and $153$ have exactly one non-trivial $1$-cycle, which has $4$ digits when written in base $b$
 (note that it follows from Theorem \ref{cor.digits} that a $1$-cycle $[n]$ for $S_{x^3,b}$   is such that  $n$ has at most $4$ digits in base $b$).
 
 Thus  Proposition \ref{pro.cubic1cycles} cannot immediately be generalized to include the case where $b=9k$. But when $b=9k^2$, 
 Proposition \ref{pro.cubic1cyclebsquare} shows that $S_{x^3,b}$ has at least six non-trivial $1$-cycles. When $9$ divides $b$, we have only succeeded in producing parametric families of $1$-cycles where $b$ is a quadratic function of $k$, as in our next  proposition.
 \end{remark}
\begin{proposition}  
 Let $b=9(730k^2-1)$. Then $n:=[27k,3k]_b=730(3k)^3$ is a $1$-cycle for $S_{x^3,b}$. 
\end{proposition}
\begin{proof} 
 An integer $n:=[x,y]_b$ is a $1$-cycle for $S_{x^3,b}$ if and only if the equation
$x^3+y^3=x+yb$ is satisfied. Looking at this equation in the form $x(x-1)(x+1)=y(b-y^2)$, we can impose that $y$ divide one of the factors $x$, $x+1$, or $x-1$,
and solve for $b:=y^2+x(x-1)(x+1)/y.$ If we want for $9$ to divide $b$, we need to impose that $y=3k$, and when we impose that $y$ divide $x$, we can take for instance $x=27k$, leading to the statement of the proposition.
\end{proof}

\begin{proposition} \label{pro.cubic1cyclebsquare} Let $k \geq 2$ be a positive integer. 
\begin{enumerate}[\rm (a)]
\item Suppose that $b=k^2$. 
Then  $[0, k]_b$ and  $[1, k]_b$ are $1$-cycles for $S_{x^3,b}$.
\item Suppose that $b=(3k+1)^2$. 
Then $[2k+1, k+1]_b$ is a $1$-cycle for $S_{x^3,b}$.
\item Suppose that $b=(3k+2)^2$. 
Then $[2k+1, k]_b$ is a $1$-cycle for $S_{x^3,b}$.
\item Suppose that $b=(3k)^2$. 
Then $[0,6k^2+k, 3k^2+2k]_b$, $[1,6k^2+k, 3k^2+2k]_b$, 
$[0,6k^2-k, 3k^2-2k]_b$ and $[1,6k^2-k, 3k^2-2k]_b$ are  $1$-cycles for $S_{x^3,b}$.
\end{enumerate}
\end{proposition} 
\begin{proof}
An integer $n:=[x,y]_b$ is a $1$-cycle for $S_{x^3,b}$ if and only if the equation
$x^3+y^3=x+yb$ is satisfied. That this is the case in (a), (b), and (c) can be checked directly. Similarly, for (d),  an integer $n:=[x,y,z]_b$ is a $1$-cycle for $S_{x^3,b}$ if and only if the equation
$x^3+y^3+z^3=x+yb+zb^2$ is satisfied.
\end{proof}

\if false
\begin{example} When $b=9$, Proposition \ref{pro.cubic1cyclebsquare} finds all but one of the seven $1$-cycles of $S_{x^3,9}$. It misses only the cycle $[n]$ with 
$n=[ 8, 8, 3, 1 ]_9=1052$. 
When $b=121$, $S_{x^3,121}$ has $88$ cycles, $50$ of which are $1$-cycles, and $5$ of these $1$-cycles need $4$ digits when written in base $b$.
\end{example}
\fi

\begin{remark} When $\phi(x)=x^2$ and $x^3$, the set of bases $b$ such that $S_{\phi,b}$ has a $1$-cycle has positive lower density (see Theorem  \ref{thm.Sub} and Proposition \ref{pro.cubic1cycles}).
We do not know if this remains the case when $\phi(x)=x^m$ and $m \geq 4$. 

When $\phi(x)=x^m$ with $m\geq 3$, we only found the following  parametric families, 
which show that the sets of integer values of certain polynomials $f(t)$ of degree $m-1$ are contained in the set of bases $b$ where $S_{x^m,b}$ has a $1$-cycle.
When $b=c^{m-1}$, then $[c^m]$ and $[1+c^m]$ are $1$-cycles for $S_{\phi,b}$. When $b=2c^{m-1}-1$, then $[c+bc]$ is a $1$-cycle.
When $b=c\frac{c^{m-1}-1}{c-1}+(c-1)^{m-1}$, then $[c+b(c-1)]$ is a $1$-cycle. When $m$ is odd, and   $b=c\frac{c^{m-1}-1}{c+1}+(c+1)^{m-1}$, then $[c+b(c+1)]$ is a $1$-cycle.
\if false
\begin{center}
    \begin{tabular}{|c|c|c|c|c|} 
      \hline
      $b$ & $c^{m-1}$ & $2c^{m-1}-1$ & $c\frac{c^{m-1}-1}{c-1}+(c-1)^{m-1}$ & $m$ odd,  $c\frac{c^{m-1}-1}{c+1}+(c+1)^{m-1}$ \\
      \hline
      $1$-cycles &  $[c^m]$, $[1+c^m]$  & $[c+bc]$ & $[c+b(c-1)]$ & $[c+b(c+1)]$ \\
      \hline   
    \end{tabular}
  \end{center}
\fi

When $m=3$, the parametrizations above produce $1$-cycles when $b=2c^2-1$, $2c^2 - c + 1$, and $2c^2 + c + 1$. Unfortunately, none of these values of $b$ are divisible by $3$.

\end{remark}

\begin{remark} \label{rem.lines} 
Let $W_1/{\mathbb Q}$ denote the algebraic surface defined by the equation $x^3+y^3-(x+by)=0$ in the affine space ${\mathbb A}^3$.
 General results on singular cubic surfaces in ${\mathbb A}^3$ predict that $W_1$ can contain at most $15$ lines of ${\mathbb A}^3$
 (use \cite{BW}, Lemma 3 (c) and \cite[page 255]{BW}).
\if false 
 Indeed, let $\overline{W}_1$ denote the projective surface in ${\mathbb P}^3$ with coordinates $(b,x,y,z)$ defined by the homogeneous cubic equation 
 $x^3+y^3-(xz^2+byz)=0$.
 The surface $\overline{W}_1$ has a single singular point at the point $P:=(1:0:0:0)$, and this singular point is an $A_2$-singularity (use \cite{BW}, Lemma 3 (c),
 and the fact that at $P$,  the surface has equation  $byz +(xz^2-x^3-y^3)=0$). According to the classification over ${\mathbb C}$ given in \cite[page 255]{BW}, we find that $\overline{W}_1$ contains $15$ lines of ${\mathbb P}^3$ over ${\mathbb C}$. 
 
 Some  of these lines are defined over ${\mathbb Q}$, such as the lines $(b,x,y)=(t,0,0)$, $(t,1,0)$, $(t,-1,0)$, 
 $(3t+1, -t,t+1)$, $(3t+1, t,-(t+1))$, and $(1,t,-t)$, and the two lines $(\zeta_3, t, -\zeta_3 t)$ and $(\zeta_3^2, t, -\zeta^2_3 t)$ are defined over the quadratic field ${\mathbb Q}(\zeta_3)$, where $\zeta_3$ is a primitive third root of unity.  The complement of $W_1$ in $\overline{W}_1$ consists of three lines meeting at the singular point. One of these lines is defined over ${\mathbb Q}$ 
 by the parametrization $(t:1:-1:0)$. 
 \fi
 Unfortunately, none of these lines produces non-trivial $1$-cycles for $S_{x^3,b}$.
 
 Let $W_1'/{\mathbb Q}$ denote the algebraic surface defined by the equation $x^3+y^3-(x+b^2y)=0$ in the affine space ${\mathbb A}^3$.
 The associated projective cubic surface in ${\mathbb P}^3$ is non-singular, and thus contains $27$ lines of ${\mathbb P}^3$ (over ${\mathbb C}$).
 Some of these lines produce the parametrizations in Proposition \ref{pro.cubic1cyclebsquare} (b), (c), and (d).
 
 \end{remark}
  
Let us now consider $2$-cycles of $S_{x^3,b}$. As noted already in Example \ref{ex.2-cyclexm}, we have 
the following parametric family:
when $b=c^4$, then ${\rm cyc}(c^3, c^9)$ is a $2$-cycle for $S_{x^3,b}$, and in this example, one of the integer
in the cycle is a $3$-digit number in base $b$, since $c^9=[0,0,c]_b$. Using this example, we find that every value of the polynomial $f(t)=t^4$ 
is among the bases $b$ such that $S_{x^3,b}$ has a $2$-cycle. The following proposition allows us to prove the same statement with a quadratic polynomial $f(t)$.

 

\if false
 The variety $W$ has dimension $3$. By solving for instance for $u$ in the first equation, $W$ can also be defined in the affine space ${\mathbb A}^4$ 
 by a single equation of degree $9$. The variety $W$ has $24$ distinct singular points over ${\mathbb C}$, each with $b^3=1$.  It has the obvious involution which sends $(x,y)$ to $(u,v)$ and fixes $b$.   It also has the involution which sends $(x,y,u,v)$ to $(-x,-y,-u,-v)$ and fixes $b$.
 \fi

\begin{proposition} \label{pro.2cycleCubic} Let $W$ denote the algebraic variety in  ${\mathbb A}^5$ 
 defined by the equations
 $x^3+y^3=u+bv$ and $ u^3+v^3=x+by.$
The variety $W$ contains the following two rational curves 
given by the parametrizations
$$
\begin{array}{lllll}
b(t):=9t^2 + 15t + 7, & x(t):=2t + 2, &y(t):=t, & u(t):=t,  
                      &  v(t):=t + 1,\\

b(t):=9t^2 + 21t + 13, & x(t):=2t + 3,& y(t):=t+1,  & u(t):=t+1, 
                       & v(t):=t + 2.
\end{array}
$$
For every integer ${\mathbf t} \geq 0$, ${\rm cyc}(n({\mathbf t}), m({\mathbf t}))$  is a $2$-cycle for $S_{x^3,b({\bf t})}$, where
$ n({\mathbf t}):= x({\bf t})+ y({\bf t})b({\bf t}) $ and $m({\mathbf t}):= u({\bf t})+ v({\bf t})b({\bf t})$.
\end{proposition} 

\begin{proof}It is straightforward to verify that $(b(t), x(t), y(t), u(t), v(t))$ verifies the equations of $W$ in both cases.
It is also clear that $0\leq x({\bf t}), y({\bf t}), u({\bf t}), v({\bf t}) < b({\bf t})$.
\end{proof}
 Note that $b(t)-1$ factorizes for both parametrizations, as $(3t+2)(3t+3)$ and $(3t+3)(3t+4)$, respectively. Thus for any integer ${\bf t}$,
$b({\bf t})$ is of the form $n(n+1)+1$ with either $n$ or $n+1$ divisible by $3$.

\if false
\begin{remark} 
 A computer search found five additional integer values of $b$ in the range $[2,1000]$ with a special $2$-cycle which do not come from the parametric families provided in  Proposition \ref{pro.2cycleCubic}, namely,
 $b=124$ with $[[ 9, 8],[ 1 ,10]]$,
 $b=364$ with  $[[  9 ,0 ],[1, 2]]$,  
 $b=492$ with  $[[  1 ,10 ],[17, 2]]$,
 $b=727$ with  $[[ 9, 0],[ 2, 1]]$ and 
 $b=928$ with  $[[ 7, 28],[ 23, 24]]$. 
 
 
 A further computer search did not find any rational curves  on $W$ passing through these points with parametric equations of the form 
 $(b_0+b_1t+b_2t^2,x_0+x_1t,y_0+y_1t,u_0+u_1t,v_0+v_1t)$.
 
 \end{remark}
\begin{remark}
The $3$-fold $W$ contains some lines of ${\mathbb A}^5$ which, unfortunately, do not seem to be related to non-trivial $2$-cycles. 
Indeed, the intersection of $W$ with the plane $y=v$ is not irreducible and one of its components $Z$ is defined by the equations $
x^3+y^3=x+by$ and $y=v$. Thus $Z$ is isomorphic to the surface $W_1$ discussed in Remark \ref{rem.lines}, and contains lines.
Similarly, the intersection of $W$ with the plane $y=-v$ is not irreducible. One irreducible component $Z'$ of this intersection
is given by  the equations $x^3+y^3=-(x+by)$ and $y=-v$.  Thus $Z'$ is isomorphic to $Z$ and also contains lines.
\end{remark}
 \fi

\end{section}

  \begin{section}{A lower bound on the number of distinct cycles of $S_{x^m,b}$}
  
  In this section, we slightly generalize Theorem 12 of H. Grundman and E. Teeple in \cite{G-T} from the case $\phi(x)=x^3$ to $\phi(x)=x^m$ for all $m\geq 3$. 
 Given positive integers $m$ and $b$, define
  \[
  N=N(m,b):= \prod_{\substack{\text{$p$ prime} \\   p-1 \mid (m-1) \\ p \le b-1 }}  p 
  \quad \cdot 
             \prod_{\substack{\text{ $p$ prime} \\ p^{r-1}(p-1) \mid (m-1) \\   p > b-1 \\ \text{ord}_p(m-1)=r-1}} p^r.
  \]
 
\begin{proposition} \label{pro.mincycle}
 Let $\phi(x) = x^m$ with $m\ge 2$.  Let $b\geq 2$. Then $S_{x^m,b}$ has at least $\gcd(b-1, N)$ distinct cycles.
 In particular, when $m \geq 5$ is prime  and    $b=mk+1$, then $S_{x^m,b}$ has at least $m$ distinct cycles.
\end{proposition}

\if false 
\begin{example} 
When $m=2$, we have $N(2,b)=2$ 
for all $b \geq 3$. Hence, when $b$ is odd, Proposition \ref{pro.mincycle} implies that $S_{x^2,b}$ has least 2 distinct cycles.
We know that $[1]$ is a cycle, and using \ref{lem:main}, we obtain that $S_{x^2,b}$ has a second cycle consisting entirely of even numbers.
The analysis of $1$-cycles in  \ref{thm.Sub} produces such a $1$-cycle $[n]$ with $n:=[(b+1)/2,(b+1)/2]_b= (b+1)^2/2$. 
\end{example}
 \fi

 To prove Proposition \ref{pro.mincycle}, we use the following slightly more general set-up.
\begin{proposition}
  \label{lem:main}
  Let $b\geq 2$ and set $B := \{0,1,\dots,b-1\}$. Let $\phi: B\to{\mathbb Z}_{\ge 0}$ be any map. Suppose that there exists 
a positive integer $a$  such that $a \mid b-1$ and such that $\phi(n) \equiv n \pmod{a}$ for all $n\in B$. Then
\[S_{\phi,b}(n) \equiv n \pmod{a} \text{ for all } n\in {\mathbb Z}_{\ge 0}.\]
  In particular,  the cycles associated to the orbits of  $n \in \{1, \dots, a\}$ under $S_{\phi,b}$ are all pairwise distinct, so that $S_{\phi,b}$ has at least $a$ distinct cycles.
\end{proposition}
\begin{proof}
  Write $n = \sum_{i=0}^dn_ib^i$ in base $b$. Then
  \[S_{\phi,b}(n) = \sum_{i=0}^d\phi(n_i) \equiv \sum_{i=0}^d n_i \equiv \sum_{i=0}^d n_ib^i = n \pmod{a}.\]
It follows that the cycles associated to the orbits of $n \in \{1, \dots, a\}$ under $S_{\phi,b}$ are all pairwise distinct.
\end{proof}

\begin{proof}[Proof of Proposition \ref{pro.mincycle}]
  Suppose that $p$ is a prime such that $p-1$ divides $ m-1$. Then for all $n\in {\mathbb Z}$, $n^m\equiv n\pmod{p}$. 
Suppose now that  $p > b-1$  and that $\varphi(p^r)=p^{r-1}(p-1)$ divides $ m-1$. Then the class of every integer $n\leq p-1$ is a unit in ${\mathbb Z}/p^r{\mathbb Z}$, and so by Euler's Theorem, $n^m\equiv n\pmod{p^r}$. It follows that $N(m,b)$ divides  $n^m- n$ for all integers in $\{0,1,\dots,b-1\}$, and we can apply
 Proposition \ref{lem:main} with $a = \gcd(b-1,N(m,b))$.
\end{proof}

\if false
\begin{remark}
If $p \le b-1$ prime, then $p\in B$ and $p^m \not\equiv p\pmod{p^r}$ for $r\ge 2$, so that is why the products of primes are separated as such.
\end{remark}
\fi

%
%
\begin{corollary} \label{cor.5cycles}
Let $\phi(x) = x^3$.  Let $k$ be any positive integer and set $b=3k+1$. Then $S_{x^3,b}$ has at least $5$ distinct cycles.
 \end{corollary}
 \begin{proof}
We know that $[1]$ is a cycle. 
Using Proposition \ref{lem:main}, we obtain that the orbit of 
$n=3$ produces a cycle  
consisting entirely of integers congruent to $0$ modulo $3$. 
Proposition \ref{pro.cubic1cycles} (a) exhibits 
three  non-trivial cycles  consisting  of integers congruent to $1$ or $2$ modulo $3$.   
\end{proof}

  \end{section}

\end{document}